\documentclass[reqno]{amsart}
\usepackage{graphicx,amsfonts,amsmath,amssymb,amscd, color}
\newcommand{\rl}{{\mathbb{R}}}
\newcommand{\cx}{{\mathbb{C}}}
\newcommand{\id}{{\mathbb{I}}}
\newcommand{\dbar}{\overline{\partial}}
\newcommand{\abs}[1]{\left|{#1}\right|}
\newcommand{\norm}[1]{\left\|{#1}\right\|}
\newcommand{\tensor}{\otimes}
\newcommand{\csor}{\widehat{\otimes}}
\newcommand{\bd}{\mathsf{b}}

\newtheorem{theorem}{Theorem}
\newtheorem{lemma}{Lemma}
\newtheorem{prop}{Proposition}
\newtheorem{cor}{Corollary}

\newtheorem*{sdt}{Serre Duality Theorem}

\DeclareMathOperator{\dm}{Dom} \DeclareMathOperator{\im}{img}

\theoremstyle{definition}

\theoremstyle{remark}

\theoremstyle{remark}

\begin{document}
\title[$L^2$ Serre Duality]{$L^2$ Serre Duality on Domains in Complex Manifolds and Applications}
\author{Debraj Chakrabarti}
\address{Department of Mathematics\\ Indian Institute of Technology Bombay\\ Powai, Mumbai --400 076\\India}
\email[Debraj Chakrabarti]{dchakrab@iitb.ac.in}
\author{Mei-Chi Shaw}
\address{Department of Mathematics\\ University of Notre Dame\\ Notre Dame, IN 46556\\ USA}
\email[Mei-Chi Shaw]{mei-chi.shaw.1@nd.edu}
\thanks{The
second-named author is partially supported by NSF grants.}
\begin{abstract}
An $L^2$ version of the Serre duality on domains  in  complex
manifolds involving duality of Hilbert space realizations of
the $\overline{\partial}$-operator is established. This duality
is used to study the solution of the
$\overline{\partial}$-equation with prescribed support.
Applications are given to $\overline{\partial}$-closed
extension of forms, as well to Bochner-Hartogs type extension
of CR functions.
\end{abstract}
\maketitle
\section{Introduction}
A fundamental result in the theory of complex manifolds is Serre's duality theorem.
This  establishes a duality between the cohomology of a
complex manifold $\Omega$ and  the cohomology of $\Omega$ with compact supports,
provided the Cauchy-Riemann operator $\dbar$ has closed range in appropriate degrees.

More precisely, this can be stated as follows: let $E$ be a holomorphic vector bundle on $\Omega$, and let $H^{p,q}(\Omega,E)$  denote the $(p,q)$-th Dolbeault
cohomology group for $E$-valued forms on $\Omega$,  and let  $H^{p,q}_{{\rm{comp}}}(\Omega,E)$  denote the $(p,q)$-th Dolbeault cohomology group with
compact support. Let $E^*$ denote the holomorphic vector bundle on $\Omega$  dual to the bundle $E$, and let $n=\dim_\cx\Omega$.
Then (we assume that all manifolds in this paper are countable at infinity):
\begin{sdt}
Suppose that each of  the two operators
\begin{equation}\label{eq-condserre} \mathcal{C}^\infty_{p,q-1}(\Omega,E)\xrightarrow{\dbar_E} \mathcal{C}^\infty_{p,q}(\Omega,E)\xrightarrow{\dbar_E}\mathcal{C}^\infty_{p,q+1}(\Omega,E)\end{equation}
 has closed range with respect to the natural Fr\'{e}chet topology.
 Then the dual of the topological vector space  $H^{p,q}(\Omega,E)$ (with the quotient Fr\'{e}chet topology)
can be canonically identified with the space $H^{n-p,n-q}_{\rm comp}(\Omega,E^*)$ with the quotient topology,
where we endow  spaces of compactly supported forms with the natural  inductive limit topology.
\end{sdt}
In fact, condition that the two maps in \eqref{eq-condserre} have closed range is also
necessary for the duality theorem to hold (see \cite{laufer}; also see \cite{LL1,LL2,LL3} for further
results of this type.)

Serre's original proof \cite{serre} is based on sheaf theory and the theory of topological vector spaces.
A different approach to this result, in the case when $\Omega$ is a compact complex manifold,
was given by Kodaira using Hodge theory (see \cite{kodaira} or    \cite{De1}.) In this note we extend Kodaira's method to non-compact Hermitian manifolds
to obtain  an $L^2$ analog of  the Serre duality. Special cases of  Serre-duality using $L^2$ methods  have appeared before
in many contexts (see \cite{kr}, or \cite[Theorem~5.1.7]{fk} and \cite{Ho2,Ho3}, for example.)  The $L^2$-Serre duality
 between the maximal and minimal realizations of the $\overline{\partial}$-operator is also used in the study of the $\dbar$-operator
 on compact complex spaces (see e.g.  \cite[Proposition~1.3]{PS}) and more general duality results (of the type discussed in $\S$\ref{sec-dualrealizations} below) are used as well in these investigations (see \cite[Chapter 5]{rup}.)
 Our treatment aims to streamline
and systematize these results, with emphasis on non-compact manifolds, and point out its close relation with the choice of $L^2$-realizations
of the Cauchy-Riemann operator $\dbar$, or alternatively, choice of boundary conditions for the $L^2$-realizations
of the formal  complex Laplacian $\dbar_E\vartheta_E+\vartheta_E\dbar_E$.

The $L^2$-duality can be interpreted in many ways. At one level, it is a duality between the standard $\Box$-Laplacian
with $\dbar$-Neumann boundary conditions, and the $\Box_c$-Laplacian with dual ( ``$\dbar$-Dirichlet")
boundary conditions. Using another approach, results regarding solution of the $\dbar$-equation in $L^2$ can be
converted to statements regarding the solution of the $\dbar_c$ equation. This leads to a solution of the $\dbar$-Cauchy
problem, i.e., solution of the $\dbar$-equation with prescribed support. At the heart of the matter
lies the existence of a duality between Hilbert space realizations of the $\dbar$-operator. This is explained
in  $\S$\ref{sec-dualrealizations}. However, for clarity of exposition, we concentrate on the classical duality between
the well-known maximal and minimal realizations of $\dbar$ in the rest of the paper.

As an application of the duality principle, we consider the problem of $\dbar$-closed extension of forms.
It is well-known that solving the $\dbar$-equation  with  a weight can be interpreted as solving $\dbar$ with bundle-valued forms (see  \cite{De2}.)  The weight function $\phi$ corresponds to the   metric for the trivial line bundle with  a metric under which the length of the vector $1$ at the point $z$ is  $e^{-\phi(z)}$.  It was used by H\"ormander  to study  the weighted  $\dbar$-Neumann operator by using weight functions  which are strictly plurisubharmonic in a neighborhood of a pseudoconvex domain. When the boundary is smooth, one  can also use the smooth weight functions to study the boundary regularity for pseudoconvex domains (see \cite{Ko2}) or pseudoconcave domains (see \cite{Sh1, Sh2}) in a Stein manifold.   In this paper we will  use the Serre duality to    study the $\dbar$ problems   with singular weight functions.
The use of singular weight functions allow us to obtain  the existence and regularity problem on  pseudoconcave domains with Lipschitz boundary  in Stein manifolds. The use of singular weights has the advantage that it  only requires  the boundary to be   Lipschitz.    Even when the boundary is smooth, the use of singular weight functions gives the regularity results much more directly (cf. the proof in \cite{Sh2} or  \cite[Chapter~9]{cs}).   This method is also useful when the manifold is not Stein,  as in the case of complex projective space $\cx\mathbb{P}^n$.  In this case, any pseudoconconvex domain  in $\cx\mathbb{P}^n$  is Stein, but $\cx\mathbb{P}^n$ is not Stein.  In recent years these problems have been studied by many people (see \cite{HI,CSW,CS1})  which are all variants of  the Serre duality results.

The plan of this paper is as follows. In $\S$\ref{sec-notation}, we recall basic definitions from complex
differential geometry and functional analysis. This material can be found in standard texts, e.g. \cite{gh,wells,grubb}.
Next, in $\S$\ref{sec-duality} we discuss  several avatars of the  $L^2$-duality theorem: at the level of Laplacians,
at the level of cohomology and for the $\dbar$ and $\dbar_c$ problems. We discuss a general form of the duality theorem using the notion of dual realizations of the $\dbar$ operator on vector bundles.
In $\S$\ref{sec-extension}, we apply the results of $\S$\ref{sec-duality} to
trivial line bundles with singular metrics on pseudoconvex domains. This leads to results on the $\dbar$-closed extension of
forms from pseudoconcave domains.  In the last section, we use the $L^2$ duality results to discuss the holomorphic extension of CR forms from the
boundary of a Lipschitz domain in a complex manifold. We obtain a proof of the Bochner-Hartogs extension theorem using duality.

\section{Notation and preliminaries}\label{sec-notation}Throughout this article,
$\Omega$ will denote a Hermitian manifold, and $E$ a
holomorphic vector bundle on $\Omega$.
\subsection{Differential operators on Hilbert spaces}\label{sec-hilbertoperators}

The metrics on $\Omega$ and $E$ induce an  inner product $(,)$ on the space
$\mathcal{D}(\Omega, E)$ of smooth compactly supported sections of $E$ over $\Omega$. The inner product is given by
\begin{equation}\label{eq-globalinnerproduct} ( f, g) = \int_\Omega \langle f,g\rangle dV,\end{equation}
where $\langle,\rangle$ is the inner product in the metric of the bundle $E$, and
$dV$ denotes the volume form induced by the metric of $\Omega$. This allows us
to define the Hilbert space $L^2(\Omega,E)$ of square integrable sections of $E$ over
$\Omega$ in the usual way as the completion of the space of smooth compactly supported sections
of $E$ over $\Omega$ under the inner product \eqref{eq-globalinnerproduct}.

Let $A$ be a differential operator acting on sections of $E$, i.e. $A:\mathcal{C}^\infty(\Omega,E)\rightarrow\mathcal{C}^\infty(\Omega,E)$,  and let $A'$ be the
formal adjoint of $A$ with respect to the inner product \eqref{eq-globalinnerproduct}. Recall that
this means that for smooth sections $f,g$ of $E$ over $\Omega$, at least one of which is
compactly supported, we have
\begin{equation} \label{eq-formaladjointdef} (Af,g)= (f,A'g).
\end{equation}
The well-known facts that $A'$ exits, that it is also a differential operator acting on sections of $E$,
and that $A'$ has the same order as $A$ follow
from a direct computation in local coordinates using integration by parts.
It is clear that $(A')'= A$, i.e. the formal adjoint of $A'$ is $A$.

By an {\em operator} $T$ from a Hilbert space $\mathsf{H_1}$ to another Hilbert space
$\mathsf{H_2}$ we mean a $\cx$-linear map from a linear subspace $\dm(T)$ of $\mathsf{H_1}$
into $\mathsf{H_2}$. We use the notation $T:\mathsf{H_1}\dashrightarrow \mathsf{H_2}$,
to denote the fact that $T$ is defined on a subspace of $\mathsf{H_1}$ (rather than on all of $\mathsf{H_1}$,
when we write $T:\mathsf{H_1}\rightarrow \mathsf{H_2}$.) Recall that such an
operator is said to be {\em closed} if its graph is closed as a subspace of the product
Hilbert space $\mathsf{H_1}\times \mathsf{H_2}$.

The differential operator $A$  gives rise to several closed operators on
the Hilbert space $L^2(\Omega, E)$.

1. The {\em weak maximal realization} $A_{\max}$:  we say for $f,g\in L^2(\Omega, E)$ that
$Af=g$ in the {\em weak sense} if for all test sections $\phi \in \mathcal{D}(\Omega, E)$ we have
that
\begin{equation}\label{eq-weak}
 (f,A'\phi)= (g,\phi).\end{equation}
(This can be  rephrased in terms of the action of $A$ on distributional sections of $E$, but we will not
need this.)
The weak maximal realization $A_{\max}$ is the densely-defined closed (cf. Lemma~\ref{lem-maxminadjoint})
linear operator on  $L^2(\Omega, E)$ with domain $\dm(A_{\max})$ consisting of
all $f\in L^2(\Omega, E)$ such that $Af\in L^2(\Omega, E)$, where $Af$ is taken in the weak sense.
On $\dm(A_{\max})$, we define $A_{\max}f=Af$ in the weak sense.

2. The {\em strong minimal realization} $A_{\min}$  is the closure of the
densely defined operator $A_{\mathcal{D}}$ on $L^2(\Omega,E)$, where $A_{\mathcal{D}}$
denotes the
restriction of $A$ to the space of compactly supported sections
$\mathcal{D}(\Omega, E)$.
More precisely, $\dm(A_{\min})$ consists of those $f\in
L^2(\Omega, E)$, for which there is a $g\in L^2(\Omega, E)$ and a
sequence $\{f_\nu\}\subset \mathcal{D}(\Omega, E)$ such that
$f_\nu\rightarrow f$ and $Af_\nu\rightarrow g$ in $L^2(\Omega, E)$.
We set $A_{\min}f=g$. The fact that $A_{\mathcal{D}}$ is closeable
is a standard result in functional analysis (see \cite{grubb}.)

More generally, a {\em closed realization} of the differential operator $A$ is a closed operator
$\tilde{A}:L^2(\Omega,E)\dashrightarrow L^2(\Omega,E)$ which extends the operator $A_{\min}$. Such an operator
satisfies
\[ A_{\min}\subseteq \tilde{A} \subseteq A_{\max}.\]
Note that if $\Omega$ is complete in its Hermitian metric (in particular if $\Omega$ is compact), then
the space $\mathcal{D}(\Omega,E)$ of compactly supported smooth sections of $E$ is dense in $\dm(A_{\max})$ in
the graph norm, and  it follows that $A_{\max}=A_{\min}$, and there is a unique closed realization of $A$
as a Hilbert-space  operator. We are more interested in the case when $\Omega$ is not complete,
e.g., when $\Omega$ is a relatively  compact domain in a larger Hermitian manifold.

We now recall the following well-known fact, which follows immediately from
\eqref{eq-weak}  (see \cite[Lemma~4.3]{grubb}):
\begin{lemma} \label{lem-maxminadjoint}
As operators on $L^2(\Omega, E)$, the weak maximal realization $A_{\max}$  of
the differential operator $A$ and the strong minimal realization $A'_{\min}$ of
its formal adjoint $A'$ are Hilbert space adjoints, i.e. we have $A_{\max} =\left(A'_{\min}\right)^*$
(note that this implies that $A_{\max}$ is closed) and also   $A'_{\min}=\left(A_{\max}\right)^*$.
\end{lemma}
\begin{proof} Let $A'_\mathcal{D}$ denote the restriction of $A'$ to the compactly supported smooth
sections $\mathcal{D}(\Omega, E)$. Then $A'_\mathcal{D}$ is a densely defined
linear operator on $L^2(\Omega, E)$ and its closure is $\overline{A'_\mathcal{D}}= A'_{\min}$.
For a fixed $f\in L^2(\Omega, E)$, consider the linear map on $\dm(A_\mathcal{D})=\mathcal{D}(\Omega, E)$
given by $\phi\mapsto (f,A'\phi)$.
The definition of $\dm(A_{\max})$ shows that this map is bounded on $\dm(A'_\mathcal{D})$ if and only if
$f\in \dm(A_{\max})$. It now follows that $(A'_\mathcal{D})^*=A_{\max}$. By taking the closure, we conclude
that $(A'_{\min})^*=A_{\max}$. Since $T^{**}=\overline{T}$ it follows that $A'_{\min}=\left(A_{\max}\right)^*$.
\end{proof}

We note parenthetically that all the definitions  and
results of this section also hold in the simpler situation when $\Omega$ is a
Riemannian manifold, and $E$ is a complex vector bundle, and are independent
of the holomorphic structure of $\Omega$ and $E$.

\subsection{Bundle-valued forms}\label{sec-bundlevalued}
We recall the  standard construction of forms
on $\Omega$ with values in $E$ . Recall that an $E$-valued $(p,q)$-form on $\Omega$ is a section of the bundle
 $\Lambda^{p,q}T^*(\Omega)\tensor E,$  where $\Lambda^{p,q}T^*(\Omega)$ is the bundle
of $\cx$-valued forms of bidegree $(p,q)$ (see \cite{wells} for details.)
We denote by $\mathcal{C}^\infty_{p,q}(\Omega,E)$ the space of $E$-valued $(p,q)$-forms of class
$\mathcal{C}^\infty$, so that if $\{e_\alpha\}_{\alpha=1}^k$ is a
local frame of $E$, then locally any element $\phi$ of
$\mathcal{C}^\infty_{p,q}(\Omega)$ has a representation
\begin{equation}\label{eq-phi}
\phi = \sum_\alpha \phi^\alpha\tensor e_\alpha,\end{equation} where
the $\phi^\alpha$ are ($\cx$-valued)  $(p,q)$-forms with smooth
coefficients.

It is well-known that the operator $\dbar$ gives rise
to an operator $\dbar\tensor \id_E=\dbar_E:
\mathcal{C}^\infty_{p,q}(\Omega,E)\rightarrow
\mathcal{C}^\infty_{p,q+1}(\Omega,E)$, via the prescription
\begin{equation}\label{eq-dbare} \dbar_E \phi = \sum_\alpha (\dbar
\phi^\alpha)\tensor e_\alpha.\end{equation} See \cite{gh} for
details of this construction. For each $p$ with $0\leq p\leq n$,
this gives rise to a complex $(\mathcal{C}^\infty_{p,*}(\Omega,E),\dbar_E)$
of $E$-valued forms on $\Omega$.

With the holomorphic vector bundle
$E\rightarrow \Omega$  we can
associate the {\em dual bundle} $E^*\rightarrow\Omega$, which is a
holomorphic vector bundle over $\Omega$, such that over a point
$x\in\Omega$, the fiber $(E^*)_x$ of $E^*$ coincides with the dual
vector space $(E_x)^*$ of  the fiber $E_x$ of $E$. One then has a
natural isomorphism of bundles $E\cong(E^*)^*$, and we will always
make this identification. If $E$ is endowed with a Hermitian bundle metric,
this induces a Hermitian bundle metric on $E^*$ in a natural way, via
the identification of $E$ and $E^*$ given by the Hermitian product on
each fiber.

We can also define a wedge product
\[ \wedge: \mathcal{C}^\infty_{p,q}(\Omega, E)\tensor
\mathcal{C}^\infty_{p',q'}(\Omega,E^*)\rightarrow
\mathcal{C}^\infty_{p+p',q+q'}(\Omega)\] of  an $E$-valued $(p,q)$-form and
an $E^*$-valued $(p',q')$-form with value an ordinary (i.e.
$\cx$-valued) $(p+p',q+q')$-form in the following way. Suppose that
$\{e_\alpha\}_{\alpha=1}^k$ is a local frame for the bundle $E$ over
some open set in $\Omega$, and let $\{f_\alpha\}_{\alpha=1}^k$ be a
frame of $E^*$. Given $\phi\in \mathcal{C}^\infty_{p,q}(\Omega,E)$ and an
$\psi\in \mathcal{C}^\infty_{p',q'}(\Omega,E^*)$, we locally write $\phi=
\sum_\alpha \phi^\alpha\tensor e_\alpha$ and $\psi=\sum_\beta
\psi^\beta\tensor f_\beta$, and define pointwise \begin{equation}
\label{eq-wedge} \phi\wedge\psi
=\sum_{\alpha,\beta}f_\beta(e_\alpha)\,\phi^\alpha\wedge\psi^\beta.\end{equation}
This extends by bilinearity to a wedge product on
$\mathcal{C}^\infty_{*,*}(\Omega,E)\tensor \mathcal{C}^\infty_{*,*}(\Omega,E^*)$.

If  $E$ is a holomorphic vector bundle on $\Omega$ define a linear operator
$\sigma_E$ on $\mathcal{C}^\infty_{*,*}(\Omega,E)$ as follows: let
$\phi$ be a form of bidegree $(p,q)$. Then we set
\begin{equation}\label{eq-sigmadef} \sigma_E\phi = (-1)^{p+q}\phi,\end{equation}
and extend linearly to $\mathcal{C}^\infty_{*,*}(\Omega,E)$.
Clearly $(\sigma_E)^2$ is the identity map on $\mathcal{C}^\infty_{*,*}(\Omega,E)$.
Further, if $T$ is any  $\rl$-linear operator from $\mathcal{C}^\infty_{*,*}(\Omega,E)$ to
$\mathcal{C}^\infty_{*,*}(\Omega,F)$ (where $F$ is another holomorphic vector bundle on $\Omega$)
of degree $d$, i.e., if for a homogeneous form $f$ we have $\deg(Tf)-\deg(f)=d$, then
we have the relation
\[ \sigma_F\,T = (-1)^d\, T\,\sigma_E.\]

 It is easy to see that the wedge product defined in
\eqref{eq-wedge} satisfies the Leibniz formula
\begin{equation}\label{eq-leibniz} \dbar(\phi\wedge\psi) = \dbar_E\phi\wedge\psi
+ \sigma_E \phi\wedge \dbar_{E^*}\psi\end{equation}

We note here that the Hermitian metric on $\Omega$ and the bundle metric on $E$ have not
played any role in this section.

\subsection{The space $L^2_*(\Omega,E)$}
We  now use the facts that  the manifold $\Omega$ has been endowed with a Hermitian metric
which we denote by $g$,
i.e., each tangent space $T_x\Omega$ has been endowed a Hermitian
inner product $g_x(\cdot,\cdot)$, which depends smoothly on the base
point $x$ and also the fact the holomorphic vector bundle $E$ has  been
endowed with a Hermitian metric $h$, i.e. for each $x\in \Omega$,
$h_x$ is a Hermitian product on the fiber $E_x$ of $E$ over $x$.
The dual bundle $E^*$ can be endowed with a Hermitian metric in the
natural way.

The bundle $\Lambda^{p,q}T^*\Omega\tensor E$ has a natural Hermitian
inner product (cf. \eqref{eq-pointwiseinnerproduct} below), so we can construct the space $L^2_{p,q}(\Omega,E)=L^2(\Omega,\Lambda^{p,q}T^*\Omega\tensor E)$
of  square integrable $E$-valued forms using the method of $\S$\ref{sec-hilbertoperators}.
We let $L^2_*(\Omega,E)$ be the orthogonal direct sum of the Hilbert spaces
$L^2_{p,q}(\Omega,E)$ for $0\leq p,q\leq n$.

We write down the pointwise inner product on the space of $E$-valued forms.
Let $\phi$ be as in \eqref{eq-phi}, and let $\psi$ be another
$(p,q)$-form with local representation
\[ \psi =\sum_\beta\psi^\beta\tensor e_\beta,\]
with respect to the same local  frame. The pointwise
inner product of the  $E$-valued $(p,q)$ forms $\phi$ and $\psi$ is given by
\begin{equation}\label{eq-pointwiseinnerproduct}
\langle\phi,\psi\rangle_x=\sum_{\alpha,\beta}\langle\phi^\alpha,
\psi^\beta\rangle_x \,h_x(e_\alpha,e_\beta)
\end{equation}
at each point $x$ in the open set where the frame
$\{e_\alpha\}$ is defined, where by
$\langle,\rangle$ on right-hand side the standard pointwise
inner-product on $\cx$-valued $(p,q)$-forms is meant (see \cite{cs}.) It is not
difficult to see that this definition is independent of the choice
of the local frame. We extend \eqref{eq-pointwiseinnerproduct} to a
pointwise inner product on $\mathcal{C}^\infty_{*,*}(\Omega,E)$ by
declaring that forms of different bidegree are pointwise orthogonal.

 \subsection{The Hodge Star}The pointwise inner product
\eqref{eq-pointwiseinnerproduct} and the wedge product
\eqref{eq-wedge} can be related by the {\em Hodge-star operator},
the map $\star_E:\mathcal{C}^\infty_{p,q}(\Omega,E)\rightarrow
\mathcal{C}^\infty_{n-p,n-q}(\Omega,E^*)$, defined by
\begin{equation}\label{eq-star}
\langle\phi,\psi\rangle dV = \phi\wedge \star_E \psi,
\end{equation}
where $dV$ is the volume form on $\Omega$ induced by the Hermitian
metric $g$. It is easy to check that \eqref{eq-star} defines
$\star_E$ as an $\rl$-linear and $\cx$-{\em antilinear map} i.e.,
for a $\cx$-valued function $f$ and a $E$-valued form $\phi$,  we
have $\star_E(f\phi)=\overline{f}\star_E\phi$.  We note that
\begin{equation}\label{eq-starstar}
\star_{E^*}\,\star_E=\sigma_E,\end{equation}
and that
\begin{equation}\label{eq-sigmastar}
\sigma_{E^*}\,\star_E = \star_E\,\sigma_E,
\end{equation}
where $\sigma_{E},\sigma_{E^*}$ are as in \eqref{eq-sigmadef}.

Let $\vartheta_E$ $:\mathcal{C}^\infty_{*,*}(\Omega,E)\rightarrow
\mathcal{C}^\infty_{*,*}(\Omega,E)$ denote the formal adjoint of $\dbar_E$,
 We recall the well-known formula
for $\vartheta_E$, and take this opportunity to point out that the
formula for $\vartheta_E$ given in the popular reference \cite[p.
152]{gh} has a typographical error.
\begin{lemma}\label{lem-vartheta}
The following formula holds:
\begin{equation}\label{eq-formaladj}
\vartheta_E=
-\star_{E^*}\,\dbar_{E^*}\,\star_E.\end{equation}\end{lemma}
\begin{proof} It is sufficient to consider the case when the smooth forms
$\phi$  and $\psi$ are  of bidegree $(p,q-1)$ and $(p,q)$
respectively and at least one of them has compact support and
compute
\begin{align*}(\dbar_E\phi,\psi)_\Omega
&=\int_\Omega\dbar_E\phi\wedge \star_E\psi\\
&=\int_\Omega\left(\dbar(\phi\wedge\star_E\psi)-\sigma_E\phi\wedge
\dbar_{E^*}\star_E\psi\right)
&\text{(using \eqref{eq-leibniz})}\\
&=- (-1)^{p+q-1}\int_\Omega \phi\wedge\dbar_{E^*}\star_E\psi
&\text{(using Stokes' formula)} \\
&= -\int_\Omega\phi\wedge
(-1)^{(n-p)+(n-q+1)}\dbar_{E^*}\star_E\psi\\
&= -\int_\Omega\phi\wedge \sigma_{E^*}\dbar_{E^*}\star_E\psi\\
&=-\int_\Omega\phi\wedge \star_E\star_{E^*}\dbar_{E^*}\star_E\psi
&\text{(using \eqref{eq-starstar})} \\
&=(\phi,-\star_{E^*}\dbar_{E^*}\star_E\psi)_\Omega.\\
\end{align*}
\end{proof}
\begin{cor}
We also have the formula
\begin{equation}\label{eq-formaladj2}
\dbar_{E}=\star_{E^*}\,\vartheta_{E^*}\,\star_E
\end{equation}
\end{cor}
\begin{proof}
Using \eqref{eq-formaladj}, we  compute
\begin{align*}
\star_E\vartheta_E\star_{E^*}&=-\star_E\star_{E^*}\dbar_{E^*}\star_E\star_{E^*}\\
&=-\sigma_{E^*}\dbar_{E^*}\sigma_{E^*}\\
&=\dbar_{E^*}.
\end{align*}
The result follows on replacing $E$ by $E^*$.
\end{proof}
\section{Duality}\label{sec-duality}
\subsection{The basic observation}\label{sec-realizations}
According to  the conventions of
multidimensional complex analysis, we adopt the following
notation: we write
\begin{center}
\begin{tabular}{lll}
$\dbar_E$  &\quad for  $(\dbar_E)_{\max}$, & the weak maximal realization
of $\dbar_E$ on $L^2_*(\Omega,E)$\\
$\dbar_{c,E}$&\quad  for  $(\dbar_E)_{\min}$, & the strong minimal realization
of $\dbar_E$ on $L^2_*(\Omega,E)$\\
$\vartheta_E$ &\quad for  $(\vartheta_E)_{\max}$, & the weak maximal realization
of $\vartheta_E$ on $L^2_*(\Omega,E)$\\
$\dbar^*_{E}$ &\quad  for  $(\vartheta_E)_{\min}$, & the strong minimal realization
of $\vartheta_E$ on $L^2_*(\Omega,E)$.\\
\end{tabular}
\end{center}
By Lemma~\ref{lem-maxminadjoint}, the operators $\dbar_E$ and $\dbar_E^*$ are Hilbert space adjoints to
each other, as are the operators $ \dbar_{c,E}$ and $\vartheta_E$.

The operator $\sigma_E$  defined in  \eqref{eq-sigmadef}  extends from the space
$\mathcal{D}_*(\Omega,E)$ of compactly supported
forms to give rise to an unitary operator on $L^2_*(\Omega,E)$. Similarly the Hodge-Star
operator  $\star_E$  defined in \eqref{eq-star} extends from $\mathcal{D}_*(\Omega,E)$
to give rise to a conjugate-linear self-isometry of $L^2_*(\Omega,E)$.  We continue to denote
these Hilbert space realizations by  $\sigma_E$ and $\star_E$ respectively. We are now ready to prove the
main observation behind the use of the Hodge-$\star$ operator in $L^2$ theory:

\begin{prop}\label{prop-basic}
Let $\Omega$ be a Hermitian manifold, and $E$ a holomorphic
vector bundle on $\Omega$ equipped with a Hermitian metric.
Let $\dbar_E, \dbar^*_E, \vartheta_{E^*}, \dbar_{c,E^*}$ be the Hilbert space
realizations as defined above, and let $f\in L^2_*(\Omega,E)$:
\begin{enumerate}
\item  $f\in\dm(\dbar^*_E)$ if and only if $\star_Ef \in \dm(\dbar_{c,E^*})$. Also on
$\dm(\dbar^*_E)$ we have the relation
\begin{equation}\label{eq-duality}
\dbar^*_E= -\star_{E^*}\dbar_{c,E^*}\star_E.
\end{equation}
\item  $f\in \dm(\dbar_E)$ if and only if $\star_E f\in \dm(\vartheta_{E^*})$. On $\dm(\dbar_E)$
we have the relation
\begin{equation}\label{eq-duality2}
\dbar_E  =\star_{E^*}\,\vartheta_{E^*}\,\star_E
\end{equation}
\end{enumerate}
\end{prop}
\begin{proof}
The results are obtained by taking the minimal and maximal realizations
of \eqref{eq-formaladj} and \eqref{eq-formaladj2} respectively.

To justify \eqref{eq-duality}, we note that if $f\in\dm(\dbar^*_E)$, there is a sequence
$\{f_\nu\}$ in $\mathcal{D}(\Omega,E)$ such that $f_\nu \rightarrow f$ in $L^2_*(\Omega,E)$ and
$\vartheta_E f_\nu\rightarrow \dbar^*_Ef$ also in $L^2_*(\Omega,E)$. Note that $\star_E f_\nu\in \mathcal{D}_*(\Omega,E^*)$, since $f_\nu$ is compactly supported. Further, since $\star_E$ extends
to an isometry of $L^2_*(\Omega, E)$ with $L^2_*(\Omega, E^*)$, it follows that
$\star_E f_\nu\rightarrow\star_E f$ in $L^2(\Omega, E^*)$. From \eqref{eq-formaladj} relating the
formal adjoints, it also follows that $\dbar_{E^*}  (\star_E f_\nu)= -(\star_{E^*})^{-1}\vartheta_E f_\nu\rightarrow
-(\star_{E^*})^{-1}\dbar^*_{E} f$. Consequently, $\star_Ef\in \dm(\dbar_{c,E^*})$, and
\eqref{eq-duality} holds. The converse assertion, that if $\star_Ef\in \dm(\dbar_{c,E^*})$ then  $f\in\dm(\dbar^*_E)$,
is proved similarly.

For \eqref{eq-duality2}, suppose that $f\in \dm(\dbar_E)$. This means that $f\in L^2_*(\Omega,E)$ and
$\dbar_E\in L^2_*(\Omega,E)$ (where $\dbar_E$ is taken in the weak sense.) Since $\star_E$ is an isometry
of the Hilbert space $L^2_*(\Omega, E)$ with the Hilbert space $L^2_*(\Omega,E^*)$, it follows that
$\star_E f\in L^2_*(\Omega,E^*)$. From \eqref{eq-formaladj2} we see that  in the weak sense, we have
$\dbar_E f= \star_{E^*} \vartheta_{E^*}(\star_E f)$. Consequently, $\vartheta_{E^*}(\star_E f)=(\star_{E^*})^{-1}\dbar_Ef \in L^2(\Omega,E^*)$. It follows that $\star_Ef\in \dm(\vartheta_{E^*})$ and \eqref{eq-duality2} holds.
The converse (if $\star_E f\in \dm(\vartheta_{E^*})$,  then $f\in \dm(\dbar_E)$ ) is proved the same way.
\end{proof}
\subsection{Duality of Laplacians}
Recall that the {\em $\dbar$-Laplacian} on $E$-valued
forms on  $\Omega$ is the  operator $\Box_E$ on $L^2_*(\Omega,E)$ defined by
\[ \Box_E=  \dbar_E\dbar_{E^*}+\dbar_{E^*}\dbar_E,\]
with domain
\[ \dm(\Box_E)=\left\{f\in L^2_*(\Omega,E)\mid f\in\dm(\dbar_E)\cap\dm(\dbar^*_E), \dbar_E f\in \dm(\dbar_E^*),\dbar_E^*f\in\dm(\dbar_E)\right\}.\]

The  {\em $\dbar_c$-Laplacian} on $E$-valued forms is the operator
\begin{align*}
 \Box_E^c &=  \dbar_{c,E}\dbar_{c,E}^*+\dbar_{c,E}^*\dbar_{c,E}\\
&=  \dbar_{c,E}\vartheta_E+\vartheta_E\dbar_{c,E}
\end{align*}
on $L^2_*(\Omega,E)$ with domain
\[ \dm(\Box_E)=\left\{f\in L^2_*(\Omega,E)\mid f\in\dm(\dbar_{c,E})\cap\dm(\vartheta_E), \dbar_{c,E} f\in \dm(\vartheta_E),\vartheta_Ef\in\dm(\dbar_{c,E})\right\}.\]
Each of $\Box$ and $\Box_E^c$ is a non-negative self-adjoint operator on $L^2_*(\Omega,E)$.
Note that on the subspace $\mathcal{D}_*(\Omega,E)$ of compactly supported $E$-valued forms
both $\Box_E$ and $\Box_E^c$ coincide with the ``formal Laplacian"
$\dbar_E\vartheta_E +\vartheta_E\dbar_E$. However, in general it is not true that
$\Box_E^c$ and $\Box_E$ are equal. By \cite[Lemma~3.1(2)]{bl}, we have $\Box_E=\Box_E^c$ if
and only if $\dbar_E=\dbar_{c,E}$. This happens if $\Omega$ is either compact or complete.

We define the spaces of $E$-valued {\em $\dbar$-Harmonic} and {\em $\dbar_c$-Harmonic} forms $\mathcal{H}_{p,q}(\Omega,E)$
and $\mathcal{H}_{p,q}^c(\Omega,E)$ by
\[ \mathcal{H}_{p,q}(\Omega,E)=\ker(\Box_E)\cap L^2_{p,q}(\Omega,E)\]
and
\[ \mathcal{H}_{p,q}^c(\Omega,E)=\ker(\Box^c_E)\cap L^2_{p,q}(\Omega,E).\]
It is easy to see that
\begin{align*}  \mathcal{H}_{p,q}(\Omega,E)&= \ker(\dbar_E)\cap\ker(\dbar_E^*)\cap L^2_{p,q}(\Omega,E)\\
&= \left\{ f\in \dm(\dbar_E)\cap \dm(\dbar^*_E)\cap L^2_{p,q}(\Omega,E)\mid \dbar_Ef=\dbar_E^*f=0\right\}.
\end{align*}
and similarly
\begin{align*} \mathcal{H}_{p,q}^c(\Omega,E)&=\ker(\dbar_{c,E})\cap\ker(\vartheta_E)\cap L^2_{p,q}(\Omega,E)\\
&= \left\{ f\in \dm(\dbar_{c,E})\cap \dm(\vartheta_E)\cap L^2_{p,q}(\Omega,E)\mid \dbar_{c,E}f=\vartheta_Ef=0\right\}.
\end{align*}
The following is now easy to prove
\begin{theorem}\label{thm-babyduality}
Let $f\in L^2_*(\Omega,E)$. Then, $f\in \dm(\Box_E)$ if and only if $\star_E f\in \dm(\Box_{E^*}^c)$.
Further, we have the relation
\begin{equation}\label{eq-laplacians}
\star_E\,\Box_E = \Box_{E^*}^c\,\star_E.
\end{equation}

Also, the restriction of the map $\star_E$ to $\mathcal{H}_{p,q}(\Omega,E)$ gives rise to an isomorphism
\begin{equation}\label{eq-babyserre}
\mathcal{H}_{p,q}(\Omega,E)\cong \mathcal{H}_{n-p,n-q}^c(\Omega,E^*)
\end{equation}
\end{theorem}
\begin{proof}
On the space
\[\left\{f\in L^2_*(\Omega,E)\mid f\in\dm(\dbar_E),\dbar_Ef\in\dm(\dbar_E^*)\right\}\]
we have, using \eqref{eq-duality} and \eqref{eq-duality2},
\begin{align*}
\dbar_E^*\dbar_E&=-\star_{E^*}\dbar_{c,E^*}\star_E\star_{E^*}\vartheta_{E^*}\star_E\\
&= -\star_{E^*}\dbar_{c,E^*}\sigma_{E^*}\vartheta_{E^*}\star_E\\
&=\star_{E^*}\sigma_{E^*}\dbar_{c,E^*}\vartheta_{E^*}\star_E.
\end{align*}
Similarly, we have on
\[\left\{f\in L^2_*(\Omega,E)\mid f\in\dm(\dbar_E^*),\dbar_E^*f\in\dm(\dbar_E)\right\}\]
the relation
\[ \dbar_E\dbar_E^*=\star_{E^*}\sigma_{E^*}\vartheta_{E^*}\dbar_{c,E^*}\star_E.\]
Combining, we have on $\dm(\Box_E)$:
\[ \Box_E =\star_{E^*}\sigma_{E^*}\Box_{E^*}^c\star_E.\]
Equation \eqref{eq-laplacians} follows on pre-composing with $\star_E$ and using \eqref{eq-starstar}.
\end{proof}
It follows that the self-adjoint operators $\Box_E$ and $\Box_{E^*}^c$ are {\em isospectral}:
a number $\lambda\in\rl$ belongs to the spectrum of $\Box_E$ if and only if $\lambda$ belongs
to the spectrum of $\Box_{E^*}^c$. Let $\{E_\lambda\}_{\lambda\in\rl}$ be
a spectral family of orthogonal projections from $L^2_*(\Omega,E)$ to itself
(cf. \cite[Chapters VII,VIII]{rn}) such that we have the spectral representation
\[ \Box_E = \int_\rl \lambda d E_{\lambda}.\]
If $\{F_\lambda\}_{\lambda\in \rl}$ is defined by
 \[F_\lambda =\sigma_{E^*}\star_E E_\lambda \star_{E^*},\]
then $F_\lambda$ is an orthogonal projection on $L^2_*(\Omega,E^*)$, and
we have the spectral representation
\[ \Box_{E^*}^c= \int_\rl \lambda d F_\lambda.\]
These statements are purely formal consequences of
\eqref{eq-laplacians}.
\subsection{Closed-range property}
In order to apply $L^2$-theory to solve the $\dbar$-equation, we first need to show that
the $\dbar$-operator has closed range. In this section we consider the consequences of this
hypothesis on the $\dbar_c$ operator.

Recall that the notation $T:\mathsf{H_1}\dashrightarrow \mathsf{H_2}$ means
that  $T$ is a linear operator  defined on a linear subspace $\dm(T)$ of $\mathsf{H_1}$ and
taking values in $\mathsf{H_2}$. Further, for notational simplicity,
we will use $\dbar_E$ to
denote the restriction $\dbar_E|_{L^2_{p,q}(\Omega)}$, when $p,q$ are given, rather than
introduce new subscripts, and adopt the same convention for $\dbar_{c,E},\vartheta_E$, and $\dbar_E^*$.
We first note the following fact
\begin{lemma}\label{lem-oplist}
If any one of operators in  the following list of  Hilbert space operators has closed range,
it follows that all  the others also have closed range:
\begin{equation}\label{eq-oplist}
\begin{cases}\dbar_E:& L^2_{p,q}(\Omega,E)\dashrightarrow L^2_{p,q+1}(\Omega,E)\\
\dbar_E^*:& L^2_{p,q+1}(\Omega,E)\dashrightarrow L^2_{p,q}(\Omega,E)\\
\dbar_{c,E^*}:& L^2_{n-p,n-q-1}(\Omega,E^*)\dashrightarrow L^2_{n-p,n-q}(\Omega,E^*)\\
\vartheta_{E^*}: & L^2_{n-p,n-q}(\Omega,E^*)\dashrightarrow L^2_{n-p,n-q-1}(\Omega,E^*)\\
\end{cases}
\end{equation}
\end{lemma}
\begin{proof}
Thanks to the well-known fact that a closed densely-defined
operator has closed range if and only if its adjoint
has closed range (see  \cite[Theorem 1.1.1]{Ho2} or \cite[Lemma~4.1.1]{cs}), it follows that $\dbar_E$ has closed range if and only if $\dbar^*_E$ has closed range,
and that $\dbar_{c,E^*}$ has closed range if and only if $\vartheta_{E^*}$ has closed range.
To complete the proof, we need  to show that $\dbar_E^*$ has closed range if and only if $\dbar_{c,E^*}$ has closed range.
Now, \eqref{eq-duality} shows that  for $f\in \dm(\dbar^*_{E})$, we have
$\norm{\dbar^*_{E}f} = \norm{\dbar_{c,E^*}(\star_E f)}$, in particular,
$f\in \ker(\dbar_E^*)$ if and only if $\star_E f\in \ker(\dbar_{c,E^*})$.
This means that the inequality $\norm{\dbar_E^*f}\geq C \norm{f}$ holds for all $f\in \ker(\dbar_E^*)^\perp$
if and only if the inequality $\norm{\dbar_{c,E^*}g}\geq C\norm{g}$ holds for all $g\in\ker(\dbar_{c,E^*})^\perp$.
Again by \cite[Lemma~4.1.1]{cs} it follows that $\dbar_E^*$ has closed range if and only if $\dbar_{c,E^*}$
has closed range.
\end{proof}

\subsection{Duality of Cohomologies} We define the {\em $L^2$ cohomology} as the quotient vector space
\[ H^{p,q}_{L^2}(\Omega,E) =\frac{ \ker(\dbar_E)\cap L^2_{p,q}(\Omega,E)}{\im(\dbar_E)\cap L^2_{p,q}(\Omega,E)},\]
Similarly, the {\em $L^2$-cohomology with the minimal realization}
is defined to the space
\[ H^{p,q}_{c,L^2}(\Omega,E) =\frac{ \ker(\dbar_{c,E})\cap L^2_{p,q}(\Omega,E)}{\im(\dbar_{c,E})\cap L^2_{p,q}(\Omega,E)}.\]
If $\dbar_E$ (resp. $\dbar_{c,E}$) has closed range, $H^{p,q}_{L^2}(\Omega,E)$ (resp. $H^{p,q}_{c.L^2}(\Omega,E)$) is a  Hilbert space with the quotient norm.

Let
\[ [\cdot]: \ker(\dbar_E)\cap L^2_{p,q}(\Omega,E)\rightarrow H^{p,q}_{L^2}(\Omega,E)\]
and
 \[ [\cdot]_c: \ker(\dbar_{c.E})\cap L^2_{p,q}(\Omega,E)\rightarrow H^{p,q}_{c.L^2}(\Omega,E)\]
denote the respective natural projections onto the quotient spaces. The  following result  was first observed by Kodaira:
\begin{lemma}\label{lem-kodaira}Let $\eta$ (resp. $\eta_c$) denote the restriction of $[\cdot]$ (resp. $[\cdot]_c$) to the
vector space of $\dbar_E$-harmonic forms $\mathcal{H}_{p,q}(\Omega,E)$ (resp. the vector space of $\dbar_{c,E}$-harmonic
forms $\mathcal{H}_{p,q}^c(\Omega,E)$.) Then

(i) $\eta$ (resp. $\eta_c$) is injective.

(ii) If $\eta$ (resp. $\eta_c$) is also surjective, then $\im(\dbar_E:L^2_{p,q-1}(\Omega,E)\dashrightarrow L^2_{p,q}(\Omega,E))$
(resp. $\im(\dbar_{c,E}:L^2_{p,q-1}(\Omega,E)\dashrightarrow L^2_{p,q}(\Omega,E))$) is closed.

\end{lemma}

\begin{proof} We write the proof only for the operator $\eta$. The proof for $\eta_c$ is similar.

 (i) Note that if $q=0$ this is obvious, since $\im\left(\dbar_E:L^2_{p,q-1}(\Omega,E)\dashrightarrow L^2_{p,q}(\Omega,E)\right)=0$. Assuming $q\geq 1$, we note that $\ker(\eta)=\ker(\dbar_E)\cap\ker(\dbar_E^*)\cap \im(\dbar_E)$, and therefore a form in $\ker(\eta)$ can be written as $\dbar g$, with $\dbar^*(\dbar g)=0$.
Then
\begin{align*}0&=(\dbar_E^*(\dbar_E g),g)\\&=\norm{\dbar g}^2.
\end{align*}

(ii) Since $\eta$ is an isomorphism, we can identify the harmonic space $\mathcal{H}_{p,q}(\Omega,E)$ with the
cohomology space $H^{p,q}_{L^2}(\Omega,E)$. Since $\mathcal{H}_{p,q}(\Omega,E)$  is a closed subspace of the Hilbert
space $L^2_{p,q}(\Omega,E)$, the space $H^{p,q}_{L^2}(\Omega,E)$ also becomes a Hilbert space. We can think of
the map $[\cdot]$ as an operator from the Hilbert space $\ker(\dbar_E)\cap L^2_{p,q}(\Omega,E)$ to the Hilbert space
$H^{p,q}_{L^2}(\Omega,E)$. Since $\eta$ is surjective, every element of $\ker(\dbar_E)$ can be written as
$f+\dbar g$, where $f\in \mathcal{H}_{p,q}(\Omega,E)$. According to the identification of
$\mathcal{H}_{p,q}(\Omega,E)$ and $H^{p,q}_{L^2}(\Omega,E)$, we have $[f+\dbar_E g]=f$. Since $\norm{f+\dbar_E g}^2=
\norm{f}^2+\norm{\dbar g}^2\geq \norm{f}^2$, so that $\norm{[f+\dbar g]}\leq \norm{f+\dbar g}$ and it follows that
$[\cdot]$ is in fact a bounded map. Therefore $\ker[.]=\im(\dbar_E)\cap L^2_{p,q}(\Omega,E)$ is closed, which was
to be shown.
\end{proof}

\begin{theorem}[$L^2$ Serre duality on non-compact manifolds]
\label{thm-serre} The following are equivalent:\\
(1) the two operators
\[ L^2_{p,q-1}(\Omega,E)\stackrel{\dbar_E}\dashrightarrow L^2_{p,q}(\Omega,E)\stackrel{\dbar_E}\dashrightarrow L^2_{p,q+1}(\Omega,E)\]
have closed range.

(2) the map $\star_E: L^2_{p,q}(\Omega,E)\rightarrow L^2_{n-p,n-q}(\Omega,E^*)$ induces
a   conjugate-linear
isomorphism of  Hilbert spaces
\begin{equation}\label{eq-serre} \tau=\eta_c\circ\star_E\circ\eta^{-1}: H^{p,q}_{L^2}(\Omega,E)\rightarrow H^{n-p,n-q}_{c,L^2}(\Omega,E^*).\end{equation}
Consequently, we can identify the Hilbert space dual of $H^{p,q}_{L^2}(\Omega,E)$ with $H^{n-p,n-q}_{c,L^2}(\Omega,E^*)$
\end{theorem}
We note here that the condition (1) is in fact the necessary and sufficient condition for the existence of
the $\dbar$-Neumann operator $\mathsf{N}_{p,q}^E$, defined as the inverse (modulo kernel) of the $\Box_E$ operator
on $(p,q)$-forms.
\begin{proof}
In the diagram
\[\begin{CD}
\mathcal{H}_{p,q}(\Omega,E)  @>\star_E>> \mathcal{H}^c_{n-p,n-q}(\Omega,E^*)\\
@V{\eta}VV  					@V{\eta_c}VV\\
H^{p,q}_{L^2}(\Omega,E) @>\tau>> H^{n-p,n-q}_{c,L^2}(\Omega,E^*)\\
\end{CD}\]
the map $\star_E$ is known to be an isomorphism from $\mathcal{H}_{p,q}(\Omega,E)$ to
$\mathcal{H}^c_{n-p,n-q}(\Omega,E)$ by Theorem~\ref{thm-babyduality} (see equation \eqref{eq-babyserre}.)
Therefore, the map $\tau$ will also be an isomorphism, if and only if, both $\eta$ and $\eta_c$ are isomorphisms. Thanks
to Lemma~\ref{lem-kodaira} this is equivalent to the two maps
$\dbar_E: L^2_{p,q-1}(\Omega,E)\dashrightarrow L^2_{p,q}(\Omega,E)$ and $\dbar_{c,E^*}: L^2_{n-p,n-q-1}(\Omega,E^*)\dashrightarrow L^2_{n-p,n-q}(\Omega,E^*)$ having closed range. Since by Lemma~\ref{lem-oplist},
the second map has closed range if and only if $\dbar_E: L^2_{p,q}(\Omega,E)\rightarrow L^2_{p,q+1}(\Omega,E)$ has closed range, the result follows.
\end{proof}
\subsection{Duality of the $\dbar$-problem and the $\dbar_c$-problem} We can use the duality principle to solve
the equation $\dbar_c u=f$, provided we know how to solve $\dbar u=f$:
\begin{theorem}\label{thm-dbarc}
Suppose that for some $0\le p\le n$ and $0\le q\le n-1$, the operator $\dbar_{E^*}:L^2_{n-p,n-q-1}(\Omega,E^*)\dashrightarrow L^2_{n-p,n-q}(\Omega,E^*)$ has closed range.  Then the range  $\im(\dbar_{c,E})\cap L^2_{p,q+1}(\Omega,E)$ is closed.
The condition that $f\in \im(\dbar_{c,E})\cap L^2_{p,q+1}(\Omega,E)$ is equivalent to
the following: for every  $g\in \ker(\dbar_{E^*})\cap L^2_{n-p,n-q-1}(\Omega,E^*)$, we have
\begin{equation}\label{eq-fwg} \int_\Omega f\wedge g =0.\end{equation}
If $\Omega$ is a relatively compact pseudoconvex domain in a  Stein manifold and $q\not=n-1$,  it is further equivalent to the condition $\dbar_{c,E}f=0$.
\end{theorem}
\begin{proof} Since $\dbar_{E^*}$ has closed range on $L^2_{n-p,n-q-1}(\Omega,E^*)$, from Hilbert space theory,
it follows that there is a bounded solution operator $K$ from $L^2_{n-p,n-q}(\Omega,E^*)$ to $L^2_{n-p,n-q-1}(\Omega,E^*)$
such that $\dbar_{E^*} K= \mathbb{I}$ (the identity map) on $\im(\dbar_{E^*})$, and $K\dbar_{E^*}= \mathbb{I}- B$,
on $\dm(\dbar_{E^*})$ where $B:L^2_{n-p,n-q-1}(\Omega,E^*)\rightarrow \ker(\dbar_{E^*})\cap L^2_{n-p,n-q-1}(\Omega,E^*)$
is the generalized Bergman projection. Set \[K_c= -\star_{E^*} K^* \star_{E},\] where $K^*$ denotes the bounded operator
from $ L^2_{n-p,n-q-1}(\Omega,E^*)$ to $L^2_{n-p,n-q}(\Omega,E^*)$ which is
the Hilbert space adjoint of the operator $K$ defined above.

Now let $f\in\im(\dbar_{c,E})\cap L^2_{p,q+1}(\Omega,E)$. Note that, this means $\star_E f\in \im(\dbar_{E^*}^*)=\ker(\dbar_{E^*})^\perp$. It follows that $B(\star_E f)=0$.

 We set $u=K_c f$. This is well-defined, since
$\star_{E}f\in L^2_{n-p,n-q-1}(\Omega,E^*)$, which is the domain of $K^*$, and we have $\norm{u}\leq C \norm{f}$.
Also, from \eqref{eq-duality} we have $\dbar_{c,E}\star_{E^*}= -\left(\star_E\right)^{-1}\dbar_{E^*}^*$. Therefore,
\begin{align*}
\dbar_{c,E}u &= -\dbar_{c,E}\star_{E^*} K^*\star_E f\\
&=\left(\star_E\right)^{-1}\dbar_{E^*}^* K^* \star_E f\\
&= \left(\star_E\right)^{-1}\left(K\dbar_{E^*}\right)^* \star_E f &\text{since $K$ is bounded}\\
&=\left(\star_E\right)^{-1}\left(\mathbb{I}-B\right)^*\star_E f\\
&= f- \left(\star_E\right)^{-1}B (\star_E f)&\text{ $B$ is self-adjoint}\\
&= f.
\end{align*}

We note that $g\in \ker(\dbar_{E^*})\cap L^2_{n-p,n-q-1}(\Omega,E^*)$ if and only if
$\star_{E^*}g\in \ker(\vartheta_{E})\cap L^2_{p,q+1}(\Omega,E)$. Since $\im(\dbar_{c,E})=\ker(\dbar_{c,E}^*)^\perp
=\ker(\vartheta_E^*)^\perp$, it follows that $f\in \im(\dbar_{c,E})$ if and only if for each $g\in \ker(\dbar_{E^*})$
we have $(f,\star_{E^*}g)=0$, i.e.,
\begin{align*}
0&=\int_\Omega f\wedge \star_E\star_{E^*}g\\
&=\int_\Omega f\wedge \sigma_{E^*}g\\
&=(-1)^{2n-p-q-1}\int_\Omega f\wedge g,
\end{align*}
which proves \eqref{eq-fwg}.

Now assume $\Omega$ is in a  Stein manifold.   Then, we know that $H^{p,q}_{L^2}(\Omega,E^*)=0$, provided $q\not=0$.
By the $L^2$ Serre duality, $H^{p,q+1}_{c,L^2}(\Omega,E)=0$, unless $q+1=n$. In other words, if $q\not=n-1$,
if $f\in \im(\dbar_{c,E})\cap L^2_{p,q+1}(\Omega,E)$, then $f\in \ker(\dbar_{c,E}:L^2_{p,q}(\Omega,E)\dashrightarrow
L^2_{p,q+1}(\Omega,E))$.  This completes the proof.
\end{proof}
\subsection{Duality of realizations of the $\dbar$ operator}\label{sec-dualrealizations}
We now discuss an abstract version of $L^2$-duality which generalizes the duality of
$\dbar_E$ and $\dbar_{c,E^*}$ discussed in the previous sections. The proofs of the
statements made below are  parallel to the proofs of corresponding statements (for $\dbar_E$ and
$\dbar_{c,E^*}$) in the previous sections, and are omitted.

Let $E$ be a vector bundle over $\Omega$ and let $D:L^2_*(\Omega,E)\dashrightarrow L^2_*(\Omega,E)$ be a
realization of $\dbar_E$, acting on $E$-valued forms.
Then $D$ satisfies $\dbar_{c,E}\subseteq D \subseteq \dbar_E.$
We define an operator $D^\vee$ on the Hilbert Space $L^2_*(\Omega,E^*)$ by setting
\[ D^\vee = \star_{E}\, D^*\, \star_{E^*},\]
where $D^*:L^2_*(\Omega,E)\dashrightarrow L^2_*(\Omega,E)$ is the Hilbert space adjoint of the operator
$D$. Then the following is easy to prove using relations \eqref{eq-formaladj} and \eqref{eq-formaladj2}:
\begin{lemma}\hfill
\begin{enumerate}
\item $D^\vee$ is a realization of the operator $\dbar_{E^*}$ on the Hilbert space $L^2_*(\Omega,E^*)$, and
its domain is $\star_E(\dm(D^*))$.
\item $(\dbar_E)^\vee= \dbar_{c,E^*}$ and $(\dbar_{c,E})^\vee =\dbar_{E^*}$.
\item The map $D\mapsto D^\vee$ is a one-to-one correspondence of the {\em closed} realizations of $\dbar_E$ with
the closed realizations of $\dbar_{E^*}$.
\end{enumerate}
\end{lemma}
We can refer to $D^\vee$ as the realization of $\dbar_{E^*}$ dual to the realization $D$ of $\dbar_E$.
From now on we will assume that the realization $D$  of the  $\dbar_E$ operator is {\em closed.} Note that then
$\ker(D)$ is a closed subspace of $L^2_*(\Omega,E)$.

We define the cohomology groups of the bundle $E$, with respect to the (closed) realization $D$ as
\[ H^{p,q}_{L^2}(\Omega,E;D) =\frac{ \ker(D)\cap L^2_{p,q}(\Omega,E)}{\im(D)\cap L^2_{p,q}(\Omega,E)}.\]
This becomes a Hilbert space if $\im(D)$ is closed in $L^2_{p,q}(\Omega,E)$

Then, we can state the following generalized version of Serre duality, with exactly the same proof:

\begin{theorem}
\label{thm-serre2} The following are equivalent for a closed realization $D$ of $\dbar_E$:\\
(1) the two operators
\[ L^2_{p,q-1}(\Omega,E)\stackrel{D}\dashrightarrow L^2_{p,q}(\Omega,E)\stackrel{D}\dashrightarrow L^2_{p,q+1}(\Omega,E)\]
have closed range.

(2) the map $\star_E: L^2_{p,q}(\Omega,E)\rightarrow L^2_{n-p,n-q}(\Omega,E^*)$ induces
a   conjugate-linear
isomorphism of the cohomology  Hilbert space $H^{p,q}_{L^2}(\Omega,E;D)$ with $H^{n-p,n-q}_{L^2}(\Omega,E^*;D^\vee)$

\end{theorem}

We give an example of a closed realization of $\dbar$ which is strictly intermediate between the maximal and minimal realizations. We consider a domain $\Omega$ in  a product Hermitian manifold $\mathcal{M}_1\times \mathcal{M}_2$, such that $\Omega$ is the product of smoothly bounded, relatively compact domains $\Omega_1\Subset \mathcal{M}_1$ and $\Omega_2\Subset\mathcal{M}_2$. We endow  $\Omega$ with the product Hermitian metric derived from $\mathcal{M}_1$ and
$\mathcal{M}_2$.

If $\mathsf{H}_1$ and $\mathsf{H}_2$ are Hilbert spaces, we denote by $\mathsf{H}_1\csor\mathsf{H}_2$ the
{\em Hilbert tensor product} of $\mathsf{H}_1$ and $\mathsf{H}_2$, i.e., the completion of the {\em algebraic}
tensor product $\mathsf{H}_1\tensor \mathsf{H}_2$ under the norm induced by the natural inner product defined on
decomposable tensors by
\[ (x\tensor y, z\tensor w)= (x,  z)_{\mathsf{H}_1}(y, w)_{\mathsf{H}_2},\]
and extended linearly. For details see \cite[$\S$3.4]{weid}.
An example of Hilbert tensor products is the space $L^2_*(\Omega)$ of square integrable forms on
the product Hermitian manifold $\Omega=\Omega_1\times\Omega_2$. In fact,
\[ L^2_*(\Omega) = L^2_*(\Omega_1)\csor L^2_*(\Omega_2),\]
if we make the natural identification $f\tensor g = \pi_1^* f\wedge \pi_2^* g$,
where $\pi_j:\Omega\rightarrow \Omega_j$ is the natural projection.

If $T_1:\mathsf{H}_1\dashrightarrow  \mathsf{H}_1'$ and $T_2:\mathsf{H}_2\dashrightarrow  \mathsf{H}_2'$
are closed densely-defined operators, we can define an operator $T_1\tensor T_2: \dm(T_1)\tensor\dm(T_2)\dashrightarrow
\mathsf{H}_1'\tensor \mathsf{H}_2'$, which on decomposable tensors takes the form $(T_1\tensor T_2)(x\tensor y)= T_1 x \tensor T_2 y$. It is well-known that provided $T_1$ and $T_2$ are closed, the operator $T_1\tensor T_2$ is closable.
The closure, denoted by $T_1\csor T_2$ is a closed densely defined operator from $\mathsf{H}_1\csor \mathsf{H}_2$
to $\mathsf{H}_1'\csor \mathsf{H}_2'$.

We let $\dbar^j:L^2_*(\Omega_j)\dashrightarrow L^2_*(\Omega_j)$
denote the maximal realization  of the $\dbar$ operator acting of $\cx$-valued forms on
$\Omega_j$. Similarly, we let $\dbar^j_c:L^2_*(\Omega_j)\dashrightarrow L^2_*(\Omega_j)$ denote the
minimal realization of the $\dbar$ operator. Consider the operator $D$  on $L^2_*(\Omega)$ defined by
\[ D= \dbar^1\csor I_2 +\sigma_1\csor \dbar^2_c,\]
where $I_2$ is the identity map on $L^2_*(\Omega_2)$ and $\sigma_1$ is the (bounded selfadjoint) operator on $L^2_*(\Omega_1)$
which when restricted to $L^2_{p,q}(\Omega_1)$ is multiplication by $(-1)^{p+q}$.  Using the techniques of \cite{ChaS,Ch}
the following properties of $D$ can be established
\begin{itemize}
\item $D$ is a closed densely-defined operator on $L^2_*(\Omega)$.
\item $D$ is a realization of $\dbar$ on $\Omega$, and it is strictly intermediate between the maximal and the minimal
realization. We may think of $D$ as being the realization which is maximal on the factor $\Omega_1$ and minimal on the factor $\Omega_2$.
\item Suppose that the maximal realization $\dbar^j$ has closed range on $L^2_*(\Omega_j)$  for $j=1$ and $2$. By duality, $\dbar_c^j$ has closed range in $L^2_*(\Omega_j)$ as well. Using either of the methods of proof used in \cite[Theorem~1.1]{ChaS} or \cite[Theorem~1.2]{Ch}, we can conclude that the operator $D$ also has closed range.
Further, we have the K\"{u}nneth formula:
\begin{align}
H^*_{L^2}(\Omega;D) &= H^*_{L^2}(\Omega_1;\dbar^1)\csor H^*_{L^2}(\Omega_2;\dbar^2_c)\nonumber\\
&=H^*_{L^2}(\Omega_1)\csor H^*_{c,L^2}(\Omega_2)\label{eq:kun}\end{align}

\item The dual realization $D^\vee$ is the one which is minimal on $\Omega_1$ and maximal on $\Omega_2$; it can be represented as
\[ D^\vee = \dbar^1_c\csor I_2 +\sigma_1 \csor\dbar^2.\]
Provided $\dbar$ has closed range in each of $\Omega_1$ and $\Omega_2$, the operator $D^\vee$ again has closed range,
and the K\"{u}nneth formula holds:
\begin{align*} H^*_{L^2}(\Omega;D^\vee)&= H^*_{L^2}(\Omega_1;\dbar^1_c)\csor H^*_{L^2}(\Omega_2;\dbar^2).\\
&=H^*_{c,L^2}(\Omega_1)\csor H^*_{L^2}(\Omega_2)
\end{align*}
Suppose that $\dim_\cx\Omega_j=n_j$, and set $n=n_1+n_2=\dim_\cx(\Omega)$. We have by Serre duality,
$H^{n-p,n-q}(\Omega;D^\vee)\cong H^{p,q}(\Omega;D)$ via the map $\star$. Note that this could also be
deduced from the knowledge of  Serre duality on the factors: indeed for each $(p_1,q_1)$, we have
$H^{p_1,q_1}_{L^2}(\Omega_1)\cong H^{n_1-p_1,n_2-q_1}_{c,L^2}(\Omega_1)$, and for each $(p_2,q_2)$ we have
$H^{n_2-p_2,n_2-q_2}_{L^2}(\Omega_2)\cong H^{p_2,q_2}_{c,L^2}(\Omega_2)$. Therefore,
\begin{align*}
H^{n-p,n-q}(\Omega;D^\vee)&=\bigoplus_{\substack{p_1+p_2=p\\q_1+q_2=q}}
\left(H^{n_1-p_1,n_2-q_1}_{c,L^2}(\Omega_1)\csor H^{n_2-p_2,n_2-q_2}_{L^2}(\Omega_2)\right)\\
&\cong\bigoplus_{\substack{p_1+p_2=p\\q_1+q_2=q}} H^{p_1,q_1}_{L^2}(\Omega_1)\csor H^{p_2,q_2}_{c,L^2}(\Omega_2)\\
&=H^{p,q}_{L^2}(\Omega;D).
\end{align*}
\end{itemize}

\section{$\dbar$-closed extension of forms}\label{sec-extension}
In this section, we assume that $\Omega$ is a relatively compact domain in a
Hermitian manifold $X$. We assume that the holomorphic vector bundle $E$ is defined on all of $X$.

\begin{prop}\label{prop-dbarc} Let $\Omega$ be a  relatively compact pseudoconvex domain with Lipschitz boundary
in a Hermitian Stein manifold $X$.
Then a form $f\in\dm(\dbar_{c,E})$ if and only if both $f^0$ and $\dbar(f^0)$ are in
$L^2_*(\Omega,E)$, where $f^0$ denotes the form obtained by extending the form $f$ by 0
on $X\setminus \Omega$. We in fact have $(\dbar_c f)^0=\dbar(f^0)$ in the distribution sense.
\end{prop}
\begin{proof}
By definition, given $f\in \dm(\dbar_{c,E})$, there is a sequence $\{f_\nu\}$ of smooth $E$-valued forms with compact support in $\Omega$ such that $f_\nu\to f$ and $\dbar f_\nu \to \dbar_c f$, both in $L^2_*(\Omega,E)$.
Then clearly $(f_\nu)^0\to f^0$ and $\dbar((f_\nu)^0)\to \dbar f$ in $L^2_*(\Omega)$.  It is also easy to see that  $\dbar((f_\nu)^0)
\to  \dbar ((f)^0)$ in  the distribution sense in $X$.  To see that $\dbar ((f)^0)=(\dbar f)^0$, we  use integration-by-parts (since $\bd\Omega$ is Lipschitz) to have that for any $\phi\in C^1_*(X)$,
\begin{align*} ((\dbar_cf)^0,\phi)_X&=(\dbar f,\phi)_\Omega\\
&=\lim_{\nu\to \infty}(\dbar f_\nu, \phi)_\Omega\\
&=\lim_{\nu\to \infty}(f_\nu, \vartheta \phi)_\Omega \\&
=(f^0, \vartheta \phi)_X\\&=(\dbar ((f)^0),\phi)_X.\end{align*}
This proves the ``only if" part of the result.

Assume now that both $f^0$ and $\dbar(f^0)$ are in
$L^2_*(\Omega,E)$. To show that $f\in \dm(\dbar_{c,E})$, we need to construct a sequence $f_\nu\in \mathcal{D}(\Omega,E)$
which converges in the graph norm corresponding to $\dbar$ to $f$. By a partition of unity, this is a
local problem near each $z\in \mathsf{b}\Omega$, and we can assume that $E$ is a trivial bundle near $z$.
By the assumption on the boundary, for any point $z\in \mathsf{b}\Omega$, there is a neighborhood $\omega$ of $z$ in $X$,
and for $\epsilon\geq 0$, a continuous one parameter family $t_\epsilon$ of biholomorphic maps from $\omega$ into $X$ such that $\Omega\cap\omega$ is compactly contained in $\Omega$, and $t_\epsilon$ converges to the identity map on $\omega$ as $\epsilon\to 0^+$. In local coordinates near $z$, the map $t_\epsilon$ is simply the translation by an amount $\epsilon$ in the inward normal direction. Then we can approximate $f^0$ locally by $f^{(\epsilon)}$, where
\[ f^{(\epsilon)}= (t_\epsilon^{-1})^* f^0\]
is the pullback of $f^0$ by the inverse $t_\epsilon^{-1}$ of $t_\epsilon$. A partition of unity argument now gives a form
$f^{(\epsilon)}\in L^2_*(X,E)$ such that $f^{(\epsilon)}$ is supported inside $\Omega$ and as $\epsilon\to 0^+$,
\[\begin{cases}
 f^{(\epsilon)}\to f^0 &\text{in $L^2_*(X,E)$}\\
\dbar f^{(\epsilon)}\to \dbar f^0 &\text{in $L^2_*(X,E)$}
\end{cases}
\]
Since $\bd\Omega$ is Lipschitz,  we can apply Friedrichs' lemma  (see \cite{Ho1} or Lemma 4.3.2 in \cite{cs}) to the form $f^{(\epsilon)}$  to construct the sequence $\{f_\nu\}$ in $\mathcal{D}(\Omega,E)$.
\end{proof}

\subsection{Use of singular weights}
Let $X$ be any Hermitian manifold, and let $\Omega\Subset X$ be a domain in $X$. We assume that
$\Omega$ is pseudoconvex, and for $z\in \Omega$, let $\delta$ be a distance function on $\Omega$. We will assume that $\delta$ satisfies the strong Oka's lemma:
\begin{equation} i\partial\dbar (-\log \delta)\ge c\omega.
\end{equation} where $c>0$ and $\omega$ is a positive (1,1)-form on $X$.

Such a distance function  always exists  on a Stein manifold.  For example, if $\Omega$ is a pseudoconvex domain  in $\cx^n$, we can take
$\delta(z)$  to be  $\delta_0e^{-t|z|^2}$ where $\delta_0$ is the Euclidean   distance from $z$ to
 to $\mathsf{b}\Omega$ and $t>0$. The distance function $\delta$ is comparable to $\delta_0$.   For each $t>0$,  let $E_t$ denote the trivial line bundle $\cx\times \Omega$
on $\Omega$ with pointwise Hermitian inner product $\langle u,v\rangle_z = (\delta(z))^t u\overline{v}$, where
$u,v\in\cx$ are supposed to be in the fiber over the point $z\in\Omega$. On a Stein manifold, we can take $\delta$ to be $\delta_0e^{-t\phi}$ for sufficiently large $t$,  where $\delta_0$ is the distance function to the boundary with respect to the Hermitian metric on $X$ and $\phi$ is a smooth strictly plurisubharmonic function on $X$.
 In classical terminology of
H\"{o}rmander, this corresponds to the use of the weight function $\phi_t = -t\log \delta$. The dual bundle $(E_t)^*$
with dual metric can be naturally identified with $E_{-t}$, i.e. the weight $t\log\delta$. We will denote
\begin{equation}\label{eq-newnotation} L^2_{p,q}(\Omega,\delta^t)= L^2_{p,q}(\Omega,E_t)\end{equation}
in conformity with the classical notation. Note that for $t>0$, the function $\delta^{-t}$ blows up at the boundary
of $\Omega$. If $t\ge  1$,  a form in $L^2_{p,q}(\Omega,\delta^{-t})$ smooth up to the boundary  {\em vanishes} on the boundary. We have the following:

\begin{prop}\label{prop-old3p2}
Let $\Omega$ be a  relatively compact pseudoconvex domain with Lipschitz boundary
in a Hermitian Stein manifold $X$ of dimension $n\ge 2$.
Suppose that
$f\in L^2_{(p,q)}(\Omega,\delta^{-t})$   for some $t\geq 0$,  where   $0\leq p\leq
n$ and
$1\leq q< n $.
Assuming that (in the sense of distributions)
$\dbar f=0$  in $X$ with $f=0$ outside $\Omega$,
then   there exists  $u_t\in L^2_{(p,q-1)}(\Omega,\delta^{-t})$  with
$u_t=0$ outside $\Omega$ satisfying
$\dbar u_t=f$ in the
distribution sense in
$X$.

  For  $q=n$,  we assume that
    $f$ satisfies
 \begin{equation}\label{eq-3p1}\int_\Omega  f\wedge g=0\qquad \text{ for every } g\in \ker(\dbar)\cap L^2_{(n-p,0)}(\Omega, \delta^{t}),\end{equation}
 the same results holds.
\end{prop}
\begin{proof} Using the notation $E_t$ as in \eqref{eq-newnotation} it follows that for any $t>0$,
the map $\dbar_{E_t^*}$ has closed range in each degree following H\"ormander's $L^2$ method \cite{Ho2} with weights  since the weight function satisfies the strong Oka's lemma (see  \cite{HaS}) This equivalent to the $\dbar$-problem  on the pseudoconvex domain $\Omega$ in the bundle $E_t^*=E_{-t}$, i.e., with plurisubharmonic weight $-t\log \delta$. The result now follows on
combining the solution of the $\dbar_c$ problem as given by Theorem~\ref{thm-dbarc} and the characterization of the $\dbar_c$ operator as given by Proposition~\ref{prop-dbarc}.
\end{proof}

For real $s$, denote by $W^s(\Omega)$ the Sobolev space of functions on $\Omega$ with $s$ derivatives in $L^2$.
Let $W^s_0(\Omega)$ be the space of completion of $C^\infty_0(\Omega)$ functions under $W^s(\Omega)$-norm.

\begin{lemma}\label{lem-old3p3} Let $\Omega$ be a bounded domain with Lipschitz boundary  in $\rl^n$ and let  $\rho$ be a distance function.  For any   $s\ge 0$,
 if $f\in W^s(\Omega)$ and $ \rho^{-s+\alpha}D^\alpha f\in L^2(\Omega)$ for every multi-integer $\alpha$ with $|\alpha|\le s$, then $f\in W^s_0(\Omega)$ and $f^0\in W^s(\rl^n)$ where $f^0$ is the extension of $f$ to be zero outside $\Omega$.
\end{lemma}
\begin{proof} When the boundary is smooth and $s$ is an integer,  this is proved in \cite[Chapter~1, Theorem~11,8]{LM}.
We first  note that when $s\leq \frac{1}{2}$, the space $W^s$ and $W^s_0$ are equal (see \cite[Chapter~1,Theorem~11.1]{LM}, or Grisvard \cite{Gr}).   When $s\neq k+\frac 12$, where $k=0, 1, 2 ,\dots$,  the lemma follows from   \cite[Section~11.2 and Theorem~11.4]{LM} for smooth domains.

   To see that when $s=k+\frac 12$ holds, we first prove for $k=0$. Let  $f\in W^{\frac 12}(\Omega)$
and $\rho^{ -\frac 12}f\in L^2(\Omega)$.   We only need to show that $f^0$ is in $W^{\frac 12}(\rl^n)$. Notice that for $0\le s\le \frac 12$,  the extension operator $u\in W^s(\Omega)=W^s_0(\Omega) \to u^0$ is continuous only when $s<\frac 12$, but is not continuous from $W^{\frac 12}(\Omega)$ to $W^{\frac{1}{2}}(\rl^n)$ (see   \cite{LM}).  However,  if $f$ satisfies $\rho^{-\frac 12} f\in L^2(\Omega)$,  then $f\in W^{\frac 12}_{00}(\Omega)$, which is a proper subset of $W^{\frac 12}(\Omega)=W^{\frac 12}_0(\Omega)$ (for definition and properties  of $W^{\frac 12}_{00}$, see  Theorem 11.7, Chapter 1 in \cite{LM}). The extension operator $f\to f^0$ is continuous from $W^s_0(\Omega)$ to $W^s (\rl^n)$ when $s=0$ and $s=1$. Thus from the interpolation theorem, it is continuous from $W^{\frac 12}_{00}(\Omega)$ to    $W^{\frac 12}(\rl^n)$  since  $W^{\frac 12}_{00}(\Omega)$  is the interpolation space of $W^0(\Omega)$ and $W^1_0(\Omega)$.  The case for $k>0$  follows from  induction.

The lemma holds for  Lipschitz domains also  since we can exhaust any Lipschitz domain $\Omega$
by smooth subdomains $\Omega_\nu$ (see Lemma 0.3 in \cite{Sh4}).  This is clear when the domain is star-shaped and the general case follows from using a partition of unity (see    \cite{Gr} for the corresponding properties for Sobolev spaces  on Lipschitz domains).

\end{proof}

Combining Proposition~\ref{prop-old3p2} and Lemma~\ref{lem-old3p3},
we  have the following regularity results on solving $\dbar$ with prescribed support.

\begin{prop}\label{prop-old3p4}  Let $\Omega\subset\subset X$
be a pseudoconvex domain with Lipschitz boundary in a Stein manifold of dimension $n\ge 3$  with a Hermitian metric.
Suppose that     $0\le p\le
n$ and
$1\le q\le  n $
and $f$ is a $(p,q)$-form with $ W^s_{0}(\Omega)\cap L^2(\Omega, \delta^{-2s})$ coefficients, where $s\ge 0$.  We assume  that
\begin{enumerate}
 \item
 for $1\leq q<n$, $f$ satisfies
  $f\in \text{Dom}(\dbar_c)$ and $\dbar_c f=0$,
\item
  for  $q=n$,
    $f$ satisfies
 \begin{equation}
\label{eq-old3p3}\int_\Omega  f\wedge g=0\qquad \text{ for every } g\in \ker(\dbar)\cap L^2_{n-p,0}(\Omega, \delta^{2s}).\end{equation}
  \end{enumerate}
Then    there exists a $(p,q-1)$-form   $u\in L^2_{p,0}(\Omega, \delta^{-2s})\cap \text{Dom}(\dbar_c)$ with $ W^s_0(\Omega) $  coefficients  satisfying $\dbar_c u=f$   in $X$.
\end{prop}

We remark that when $s-\frac{1}{2}$ is not a non-negative integer, the assumption $f\in W^{s}_0(\Omega)$ implies that $f\in L^2(\Omega, \delta^{-2s})$ (see \cite{LM}). The pairing in \eqref{eq-old3p3} is well-defined between the two spaces  $L^2(\Omega,  \delta^{2s})$ and $L^2(\Omega, \delta^{-2s})$.

\begin{theorem}\label{thm-old3p5}Let $X$ be a Stein manifold and let $\Omega\subset \subset X$
be a relatively compact pseudoconvex  domain with Lipschitz boundary. Let $\Omega^+=X\setminus \Omega$.

Then for any $f\in W^s_{p,q}(\Omega^+)$, where
$q\leq n-2$, with $s\geq 1$ such that $\dbar f=0$ in $\Omega^+$ there exists $F\in W^{s-1}_{p,q}(X)$ with
$F|_{\Omega^+}=f$ and $\dbar f=0$ on $X$.

For $q=n-1$, we assume that
\begin{equation}\label{eq-old3p4} \int_{\bd\Omega} f\wedge g=0\quad\quad \text{for every $g\in \ker(\dbar)\cap L^2_{n-p,0}(\Omega,\delta^{2(s-1)})$},\end{equation}
and the same conclusion holds.
\end{theorem}
\begin{proof}
Since $\Omega$ has Lipschitz  boundary, there is a bounded extension
operator from
   $W^s(\Omega^+)$ to $W^s({} X)$ for all $s\ge 0$ (see e.g.
\cite{Gr}).  Let
$\tilde f\in W^{s}_{p,q}({} X)$  be the  extension of
$f$  so that
$ \tilde f|_{ {\Omega}^+  } = f $
with
   $\|\tilde f\|_{W^{s}({} X)}\le
C\|f\|_{W^{s}(\Omega^+)}.$
We have    $\dbar \tilde f\in
  W^{s-1}_0(\Omega)\cap
L^2(\Omega, \delta^{-2(s-1)})$ (see Theorem 11.5 in \cite{LM}).

Obviously we have  that $\dbar\tilde f\in W^{s-1}_0(\Omega)$ is $\dbar$-closed in $\Omega$. When $q=n-1$, $\dbar \tilde f \in W^{s-1}_{p,n}(\Omega)\cap L^2_{p,n}(\Omega, \delta^{-2(s-1)})$ and  satisfies
\begin{equation}\label{eq-old3p5}\int_{\Omega} \dbar\tilde   f\wedge g=\int_{\bd\Omega} f\wedge g=0\qquad \text{ for every } g\in \ker(\dbar)\cap L^2_{n-p,0}(\Omega, \delta^{2(s-1)}).\end{equation}
Notice that both  integrals in \eqref{eq-old3p5} are well-defined by an approximation arguments using
Friedrichs' lemma (see \cite{Ho1} or Lemma 4.3.2 in \cite{cs}).

Let $t=s-1\ge 0$.    We define $T\tilde f$   by
 $T\tilde f = - {\star_{(2t)}\bar\partial N_{2t}
{(\star_{(-2t)}  \dbar \tilde
f)}}$ in
$ \Omega$, where $\star_t= \star_{E_t}$.
      From Proposition~\ref{prop-old3p2} and
Proposition~\ref{prop-old3p4}, we have that there exists  $u=T\tilde f\in
L^2(\Omega,\delta^{-2t})\cap W^{t}_0(\Omega)$ satisfying     $\dbar
(T\tilde f)^0=\dbar \tilde f\quad \text {in }  {} X $.

Define
\[ F=\tilde f-(T\tilde f)^0=\begin{cases} f,&\quad x \in
\overline{\Omega}^+,\\\tilde f-T\tilde f ,&\quad x\in
\Omega.\end{cases}\]
Then from Lemma~\ref{lem-old3p3},  $F\in W^{s-1}_{p,q}({} X)$ and $F$ is a $\dbar$-closed
extension of $f$.\end{proof}
\begin{cor}\label{cor-old3p6} Let $\Omega_1$ and $\Omega$ be two pseudoconvex domains in a Stein manifold $\chi$ with $\Omega\subset\subset \Omega_1\subset\subset \chi$.   Let    $ \Omega^+=\Omega_1\setminus\overline\Omega $ be the  annulus between two pseudoconvex domains $\Omega$ and $\Omega_1$.    For any $f\in W^{s}_{p,q}({\Omega}^+)$, where $0\leq p\leq n$, $1\leq q< n-1$  and
 $s\geq 1$,  such that  $\dbar f = 0$ in $\Omega^+$, there exists
$u\in W^{s}_{(p,q-1)}(\Omega^+)$ with
$\dbar u=f$ in $\Omega^+$. Furthermore, if $f\in C^\infty_{p,q}(\overline \Omega^+)$, we have   $u\in C^\infty_{p,q-1}(\overline \Omega^+)$.

When $q=n$, we assume that $f$ satisfies \eqref{eq-old3p4} instead, then the same result holds.

\end{cor}

 We remark that  Corollary~\ref{cor-old3p6} allows us to solve $\dbar$
smoothly up to the boundary  on pseudoconcave domains with only Lipschitz boundary
provided the compatibility conditions are satisfied. Results of this kind was obtained in \cite {Sh1} for pseudoconcave domains with smooth boundary.  For Lipschitz boundary, see \cite{MS}  or \cite{HI} using integral kernel methods.  This is in sharp contrast of pseudoconvex domains, where solving $\dbar$ smoothly up to the boundary is known only for pseudoconvex domains with smooth boundary (see   \cite{Ko2}) or domains with Stein neighborhood basis (see  \cite{Du}).
 If the boundary $\mathsf{b}\Omega$ is smooth, Theorem~\ref{thm-old3p5} and Corollary~\ref{cor-old3p6}
also hold for $s=0$ (see \cite{Sh2, Sh4}).

\section{Holomorphic extension of CR   forms from the boundary of a complex manifold}

In this section we study holomorphic extension of CR forms from the boundary of a domain  in a complex manifold $X$ using our $L^2$-duality. The use of duality in the study of holomorphic  extension of CR functions with smooth or continuous data is classical (see \cite{serre1}),
and has been studied by many authors (see \cite{serre,kr,hl}.)

In what follows, $X$ is a complex manifold, and $\Omega$ is a relatively compact domain in $X$ with Lipschitz boundary (see
\cite{Sh4} for a general discussion of partial differential equations on Lipschitz domains, and \cite{Sh5} for a discussion of the tangential Cauchy-Riemann equations.)  We will assume that
$X$ has been endowed with a Hermitian metric, and the spaces $L^2_{p,q}(\Omega)= L^2_{p,q}(\Omega,\cx)$ of
square integrable forms are defined with respect to the metric of $X$ restricted to $\Omega$. Observe that
the spaces $L^2_{p,q}(\Omega)$ as well as the Sobolev spaces of forms $W^k_{p,q}(\Omega)$ are defined
independently of the particular
choice of metric on $X$.  Further, it is possible to define Sobolev spaces on the boundary $\bd \Omega$ in such a way
that the usual results on existence of a trace still holds, e.g. functions in $\Omega$ of class $W^1(\Omega)$ have traces
on $\bd\Omega$ of class $W^{\frac{1}{2}}(\bd\Omega)$ (see \cite{jk,ken}.)

The main observation, which follows from the duality results in \S\ref{sec-duality} is the following:

\begin{prop}\label{prop-closedrange}
For any $p$, with $0\leq p\leq n$,
the map
\[ \dbar_{c}: L^2_{p,0}(\Omega)\dashrightarrow L^2_{p,1}(\Omega)\]
has closed range.
\end{prop}
\begin{proof} Thanks to Lemma~\ref{lem-oplist} this is equivalent to the map
$\dbar: L^2_{n-p,n-1}(\Omega)\dashrightarrow L^2_{n-p,n}(\Omega)$ having closed range.
But it is well-known that $\dbar$ has closed range in this top degree on smooth domains, a fact that
is equivalent to the solvability of the Dirichlet problem for the Laplace-Beltrami operator on such domains (see \cite{fk}.)
For a proof of the solvability of the Dirichlet problem for domains  with Lipschitz boundary, see \cite{jk, ken}.
\end{proof}

Recall that a holomorphic $p$-form is a $\dbar$-closed $(p,0)$-form. We denote the space
of holomorphic $p$-forms on $\Omega$ by $\mathcal{O}_p(\Omega)$.
We deduce a necessary condition for a $(p,0)$-form on $\bd\Omega$ to be the boundary value of a holomorphic $p$-form on
$\Omega$:

\begin{theorem} \label{thm-compatibility}
Let $f\in W^{\frac{1}{2}}_{p,0}(\bd\Omega)$ be a $(p,0)$ form on $\bd\Omega$ with coefficients in the Sobolev
space $W^\frac{1}{2}$. Then the following are equivalent:
\begin{enumerate}
\item There is a holomorphic $p$-form $F\in \mathcal{O}_p(\Omega)\cap W^1(\Omega)$ such that $f=F|_{\bd\Omega}$
\item For all $g\in L^2_{n-p,n-1}(\Omega)\cap \ker(\dbar)$, we have
\begin{equation}\label{eq-5p3}
\int_{\bd\Omega}f\wedge g=0.
\end{equation}
{\em (Note that it is easy to show that a $\dbar$-closed form with $L^2$ coefficients has a trace of class $W^{-\frac{1}{2}}$, and
hence the integral above is well defined.)}
\item For any extension $\tilde{f}\in W^1_{p,0}(\Omega)$ of $f$ to $\Omega$ as a $(p,0)$-form with
coefficients in $W^1$, the form $\dbar\tilde{f}\in L^2_{p,1}(\Omega)$ belongs to the range of $\dbar_c$ on $\Omega$.
\end{enumerate}
\end{theorem}
 \begin{proof}
$(1\implies 2)$  Let $g\in L^2_{n-p,n-1}(\Omega)\cap\ker(\dbar)$. By Stoke's Theorem:
\[ \int_{\bd\Omega}f\wedge g=\int_\Omega d(F\wedge g)=\int_\Omega  \dbar (F\wedge g)=0.\]

$(2\implies 3)$ First note that such an extension $\tilde{f}$ always exists, since $\bd\Omega$ is Lipschitz. Again let $g\in L^2_{n-p,n-1}(\Omega)\cap\ker(\dbar)$. By Stoke's Theorem
\[ \int_\Omega \dbar\tilde f\wedge g = \int_{\bd\Omega} f\wedge g=0.\]
Assertion (3) now follows from the condition \eqref{eq-fwg} given in Theorem~\ref{thm-dbarc}
for a form to be in the range of the $\dbar_c$ operator.

$(3 \implies 1)$ By Proposition~\ref{prop-closedrange}, $\dbar_c$ has closed range in degree $(p,1)$, and
by hypothesis $\dbar\tilde{f}$ is in the range of $\dbar_c$. By Theorem~\ref{thm-dbarc}, we can solve the equation
\begin{equation}\label{eq-dbarcext} \dbar_c u= \dbar\tilde{f},\end{equation}
with $L^2$ estimates for a $(p,0)$-form $u$. Then $F=\tilde{f}-u$ is holomorphic in $\Omega$. Also, by Proposition~\ref{prop-dbarc}
we have that
\[ \dbar(u^0)=(\dbar u)^0=(\dbar\tilde{f})^0,\]
where the $g^0$ denotes the extension of the form $g$ on $\Omega$ to all of $X$ by setting it equal
to 0 on $X\setminus \Omega$. Since $(\dbar\tilde{f})^0\in L^2_{p,1}(X)$, by elliptic regularity, $u^0\in W^1_{p,0}(X)$.
It follows that $u^0$ has a trace (of class $W^\frac{1}{2}(\bd\Omega)$) on the Lipschitz hypersurface $\bd\Omega$.
Since $u^0$ vanishes identically on $X\setminus \Omega$, it follows that this trace is 0. Consequently,
$F\in W^1_{p,0}(\Omega)$ and satisfies $F|_{\bd\Omega}=f$.
\end{proof}
Let $f$ be a $p$-forms with coefficients  in  $L^1(\bd\Omega)$ which  is the boundary value of a holomorphic
$p$-form  $F \in \mathcal{O}_p(\Omega)$, then $f$ must be CR, i.e, it must  satisfy in the homogeneous tangential Cauchy-Riemann equations
on $\bd\Omega$ in the weak sense, i.e., for each compactly supported smooth $(n-p,n-2)$-form
$\phi\in \mathcal{D}_{n-p,n-2}(X)$,
we have
\begin{equation} \label{eq-5p2}\int_{\bd\Omega} f\wedge \dbar \phi=0.\end{equation}
(See \cite{Sh6} for details.)

It is easy to see that \eqref{eq-5p3} implies \eqref{eq-5p2}. But in general, the two conditions are not equivalent.  One
condition under which they are equivalent is the following:

\begin{cor}\label{cor-5p2} Let $\Omega$ be a domain  with Lipschitz  boundary in a complex
  manifold $X$ of complex dimension $n\geq {2}$. Suppose
that  $ H^{n-p,n-1}_{L^2}(
\Omega)=0$.       Then every  CR  form in  $f\in W^{\frac{ 1}{2}}_{p,0}(\bd\Omega)$  has  a holomorphic extension $F$ to $\Omega$ with $F\in \mathcal{O}_p(\Omega)\cap W^1(\Omega)$ and $F=f$ on $\bd\Omega$.

\end{cor}


\begin{proof} Let $g\in \ker(\dbar)\cap L^2_{n-p,n-1}(\Omega) $. By the hypothesis on cohomology, there is a
$u\in\dm(\dbar)\cap L^2_{n-p,n-2}(\Omega)$, such that $\dbar u =g$. Since $\Omega$ is Lipschitz, by Friedrich's lemma, we can find
a sequence $\{u_\nu\}\subset \mathcal{C}^\infty_{n-p,n-2}(\overline{\Omega})$ such that $u_\nu\to u$ in $L^2_{n-p,n-2}(\Omega)$,
and $\dbar u_\nu\to g$ in $L^2_{n-p,n-1}(\Omega)$ as $\nu\to \infty$. Let $\phi_\nu\in\mathcal{D}_{n-p,n-2}(X)$ be a smooth compactly supported extension of the form $u_\nu$ to $X$. Then we have
\[ \int_{\bd\Omega}f\wedge g = \lim \int_{\bd\Omega} f\wedge \dbar\phi_\nu=0.\]
The result now follows by Theorem~\ref{thm-compatibility}.
\end{proof}

Another extension result that can be deduced from Theorem~\ref{thm-compatibility} :

\begin{cor}\label{cor-5p3} Let $\Omega\Subset X$ be  a domain with connected  Lipschitz  boundary in a non-compact connected  complex
 manifold $X$  of complex dimension $n\geq {2}$. Suppose that   there exists a relatively compact  domain $\Omega'$ with Lipschitz boundary such that  $\Omega\Subset \Omega'\Subset X$ and

\begin{equation}\label{eq-5p5} H^{n-p,n-1}_{L^2}(\Omega') =0. \end{equation}
  Then every  CR form of degree $(p,0)$ on $\bd\Omega$ of Sobolev class  $ W^{\frac{1}{2}}(\bd\Omega)$
 has  a holomorphic extension to $\Omega$ (of class $W^1(\Omega)$.)
\end{cor}

\begin{proof} Let $\tilde{f}$ be an extension of $f$ to $\Omega$ (of class $W^1(\Omega)$) and let
\[ g = \begin{cases} \dbar \tilde{f} & \text{on $\Omega$}\\
0 & \text{on $\Omega'\setminus \Omega$}
\end{cases}
\]

We claim that $\dbar g=0$ on $\Omega'$. Indeed, let $u\in\mathcal{D}_{p,1}(\Omega')$ be a
smooth $(p,1)$ form of compact support in $\Omega'$.  We have
\begin{align*}
(\dbar g,u)_{L^2(\Omega')} &= (g,\vartheta u)_{L^2(\Omega')}\\
&=(\dbar\tilde{f},\vartheta u)_{L^2(\Omega)}\\
&=\int_{\Omega}\dbar\tilde{f}\wedge  \star\vartheta u\\
&=\int_{\Omega}\{\dbar(\tilde{f}\wedge \star \vartheta u) - (-1)^p (\tilde{f}\wedge\dbar\star\vartheta u)\}.
\end{align*}
Since $\dbar\star\vartheta= -\dbar\star(\star\dbar\star)= \pm\dbar\dbar\star=0$, the second term vanishes,
and by Stoke's theorem,  the first integral is equal to
\begin{align*}\int_{\bd\Omega}\tilde{f}\wedge \star \vartheta u&=
\pm\int_{\bd\Omega} {f}\wedge (\star\vartheta\star)(\star u)\\&
=\pm\int_{\bd\Omega}{f}\wedge\dbar(\star u),\\
&\text{(since $\dbar=\star\vartheta\star$ on compactly supported forms, see \eqref{eq-formaladj2})}\\
&=0,\\
&\text{(since $f$ is CR, see \eqref{eq-5p2}).}\end{align*}

As $g$ vanishes near $\bd\Omega'$ and $\dbar g=0$, it follows that $g\in\dm(\dbar_c)$ on $\Omega'$ and
$\dbar_c g=0$. Since $\dbar$ has closed range in $\Omega$ for bidegrees $(n-p,n-1)$ as well as $(n-p,n)$ it follows
by duality from  \eqref{eq-5p5} that $H^{p,1}_{c,L^2}(\Omega')=0$. There is then a $u\in \dm(\dbar_c)$ such that
$\dbar_cu =g$.  By Proposition~\ref{prop-dbarc},  the extensions by 0 satisfy $\dbar(u^0)= (\dbar u)^0= g^0$. Since $g^0$
is in $L^2(X)$ it follows that $u^0\in W^1_{p,0}(X)$. Further, $u^0$ is holomorphic on $X\setminus \Omega$ and $u^0\equiv 0$
on $X\setminus\Omega'$. By analytic continuation, $u^0\equiv 0$ on $X\setminus \Omega$. Therefore, the trace of $u$
$\bd\Omega$ vanishes, and the form $F= \tilde{f}-u$ on $\Omega$ is holomorphic, of class $W^1$ and satisfies $F=f$ on
$\bd\Omega$.
\end{proof}
\begin{cor} Let $\Omega$ be  domain with Lipschitz  boundary in a Stein  manifold $X$  of complex dimension $n\geq {2}$. Suppose
that $\bd\Omega$ is connected.
 Then for  every  CR function on $\bd\Omega$ of class $ W^{\frac{1}{2}}(\bd\Omega)$
 has  a holomorphic extension to $\Omega$.
\end{cor}
\begin{proof} In the proof of Corollary~\ref{cor-5p3}, we  let $\Omega'$ be some strongly pseudoconvex domain  in $X$  and $\Omega\Subset \Omega'$.  Then    $  H^{n,n-1}_{L^2}(\Omega')= H^{0,1}_{c,L^2}(\Omega') =0$.   The corollary follows.
\end{proof}

 When $X=\cx^n$ and $p=0$, this gives the usual Bochner-Hartogs' extension theorem.
In this case, the extension function can be written explicitly as

\[\label{eqref-5p6} F(z)= \int_{\bd\Omega} B(\zeta,z)\wedge f(\zeta), \quad   z\in \Omega,\]
where $B$ is the Bochner-Martinelli kernel.  The function $F$ has boundary value $f$ as $z$ approaches the boundary (see
\cite{straube} for a proof when the boundary is smooth; in this case we can allow more singular boundary values than
possible in our results with Lipschitz boundaries.)
This is very different from holomorphic extension of  CR functions in complex manifolds which are not Stein.  We will give an example to show that   the extension results on Lipschitz domain is  maximal in the sense that the results might  not hold if the Lipschitz condition is dropped.

We will analyze the holomorphic extension of functions  on a non-Lipschitz domain.
Let $\Omega$ be the  Hartogs' triangle in $\cx\mathbb{P}^2$ defined by
      \[ \Omega=\{[z_0,z_1,z_2]\mid |z_1|<|z_2|\},\]
where $[z_0,z_1,z_2]$ denotes the homogeneous coordinates of a point in $\cx\mathbb{P}^2$.
As usual we endow $\Omega$ with the restriction of the  Fubini-Study metric of $\cx \mathbb{P}^2$.

 \begin{prop} Let $\Omega\subset\cx\mathbb{P}^2$ be the Hartogs' triangle.
Then we have the following:
\begin{enumerate}
\item The Bergman space of $L^2$ holomorphic functions   $L^2(\Omega)\cap\mathcal{O}(\Omega)$
on the domain $\Omega$ separates points in $\Omega$.

\item There exist nonconstant  functions in the space $W^1(\Omega)\cap\mathcal{O}(\Omega)$.
However, this space does not separate points in $\Omega$ and  is not dense in the Bergman space $L^2(\Omega)\cap \mathcal{O}(\Omega)$.
\item  Let $f\in W^2(\Omega)\cap \mathcal{O}(\Omega)$ be a holomorphic function on $\Omega$ which is in the
Sobolev space $ W^2(\Omega)$. Then $f$ is a constant.
  \end{enumerate}
\end{prop}
{\em Remark:}  Statements (1) and (3) above have already been  proved in \cite{HI}.
Regarding (2), we would like to point out a misleading statement made
in that paper, where it  is claimed that  $W^1(\Omega)\cap\mathcal{O}(\Omega)$
consists of constants only (see item 5 in Example 12.1 in \cite{HI}).
\begin{proof} For (1), consider the two holomorphic functions $\frac{z_1}{z_2}$ and $\frac{z_0}{z_2}$
on $\Omega$, which separate points on $\Omega$ and the first of which is bounded (and therefore square-integrable
in the Fubini-Study metric) on $\Omega$.
To see that    $\frac{z_0}{z_2}$ is in    $L^2(\Omega)\cap \mathcal{O}(\Omega)$, we only need  to verify that it is in $L^2(\Omega)$ near the point $[1,0,0]$. We choose coordinate chart $U_0=\{z_0\not=0\}\cap\Omega$ for $\Omega$ with holomorphic coordinates
$(z,w)$, where $z=\frac{z_1}{z_0}$ and $w=\frac{z_2}{z_0}$.  The function $\frac{z_0}{z_2}=w^{-1}$ and it suffices to show that $w^{-1}$ is square-integrable on $\Omega\cap P$  where $P$ is the polydisc $\{\abs{z}<1,\abs{w}<1\}$.
More generally, consider the square-integrability of $w^{-\nu}$, where $\nu\geq 1$ is an integer. We have
\begin{align*} \int_{\Omega\cap P}  \frac 1{|w^\nu|^2}dV&= 4\pi^2 \iint_{r_1<r_2<1}\left(\frac {1}{r_2^{2\nu}}\right)r_2dr_2r_1dr_1\\
&=4\pi^2\int_0^1\left(\int_{r_1}^1 r_2^{-2\nu+1}dr_2\right)r_1dr_1
\end{align*}
When $\nu=1$ the integral becomes
\begin{align*}
\phantom{\int_{\Omega\cap P}  \frac 1{|w^\nu|^2}dV}&=4\pi^2\int_0^1-r_1\log r_1dr_1\\&<\infty.
\phantom{\pi \iint_{r_1<r_2<1}\left(\frac {1}{r_2^{2\nu}}\right)r_2dr_2r_1dr_1}
\end{align*}
If $\nu>1$, the inner integral evaluates to a constant times $(1-r_1^{-2\nu+2})$, the double integral diverges,
and consequently, $w^{-\nu}\not\in L^2(\Omega\cap P)$ (cf. \cite[Proposition~3]{HI}.)

On the subset $\Omega\cap\{z_2\not=0\}$, introduce the coordinates $\widetilde{z}=\frac{z_1}{z_2}$ and
$\widetilde{w}=\frac{z_0}{z_2}$. In these coordinates the set $\Omega\cap\{z_2\not=0\}$ is represented
as the bidisc with one infinite radius $\{(\widetilde{z},\widetilde{w})\mid \abs{\widetilde{z}}<1\}$,
and any function $f\in\mathcal{O}(\Omega)$ has a power series expansion on this polydisc of the
form
\[ f(\widetilde{z},\widetilde{w})=\sum_{\substack{\mu\geq 0\\\nu\geq 0}}C_{\mu,\nu}\widetilde{z}^\mu\widetilde{w}^\nu.\]
In the coordinate patch $\Omega\cap\{z_0\not=0\}$, the natural coordinates are $(z,w)$, where $z=\frac{z_1}{z_0}=\frac{\widetilde{z}}{\widetilde{w}}$ and $w=\frac{z_2}{z_0}=\frac{1}{\widetilde w}$.
Therefore the holomorphic function $f$ on
on $\Omega$ has a Laurent expansion on $\Omega\cap\{z_0\not=0\}$ of the form
\[ f(z,w)= \sum_{\substack{\mu\geq 0\\\nu\geq 0}}C_{\mu,\nu}\left(\frac{z}{w}\right)^\mu w^{-\nu}.\]

By the symmetry of the Fubini-Study metric, it follows that the terms of the series are orthogonal, provided they
are in $L^2(\Omega\cap P)$, and therefore, if $f\in L^2(\Omega\cap P)$,  we have
\[\norm{f}_{L^2(\Omega\cap P)}^2= \sum_{\mu,\nu\geq 0} \abs{C_{\mu,\nu}}^2\norm{\left(\frac{z}{w}\right)^\mu w^{-\nu}}_{L^2(\Omega\cap P)}^2.\]
Since $\frac{z}{w}=\frac{z_1}{z_2}$ is bounded the computation of $\norm{w^{-\nu}}_{L^2}$ in the last paragraph shows that
nonzero terms on the right hand side are not in $L^2$ if $\nu\geq 2$, which means $C_{\mu\nu}=0$ if
$\nu\geq 2$. Thus
each $f\in L^2(\Omega)\cap \mathcal{O}(\Omega)$ has a  Laurent expansion of the form
\begin{equation}\label{eq-laurent}  f(z,w)= \sum_{\substack{\mu\geq 0\\0\leq \nu\leq 1 }}C_{\mu,\nu}\left(\frac{z}{w}\right)^\mu w^{-\nu}.
\end{equation}

Taking a derivative we see that
\[ \frac{\partial f}{\partial w}(z,w)=
\sum_{\substack{\mu\geq 0\\0\leq \nu\leq 1 }}-(\mu+\nu)C_{\mu,\nu}\left(\frac{z}{w}\right)^\mu w^{-(\nu+1)}.\]
By orthogonality of the terms again, if this is in $L^2(\Omega\cap P)$ the coefficients $C_{\mu,1}=0$.  It follows
that any $f\in W^1(\Omega)\cap L^2(\Omega)$ is of the form
\begin{equation}\label{eq-w1} f(z,w)=\sum_{\nu=0}^\infty b_{\nu}\left(\frac{z}{w}\right)^\nu.
\end{equation}
Further, it is easily verified that if $f$ is of the above form then $\frac{\partial f}{\partial z}\in L^2(\Omega)$.
Therefore any holomorphic function in $W^1(\Omega)$ is a function of $\frac zw$ alone,  and it follows that $W^1(\Omega)\cap\mathcal{O}(\Omega)$  does not separate points in $\Omega$. This proves (2).

By taking two derivatives in \eqref{eq-w1}, we obtain
\[ \frac{\partial^2 f}{\partial w^2}(z,w)=
 \sum_{\nu=1}^\infty -\nu(\nu+1)b_{\nu}\left(\frac{z}{w}\right)^\nu\cdot\frac{1}{w^2}.\]
None of the mutually orthogonal terms is in $L^2(\Omega\cap P)$, thanks to the computation of $\norm{w^{-\nu}}_{L^2}$ above.
It follows that $f$ reduces to a constant and we have (3).

\end{proof}

\end{document}